\documentclass[reqno,12pt,letterpaper]{amsart}
\usepackage{amsmath,amssymb,amsthm,graphicx,mathrsfs,url}
\usepackage[usenames,dvipsnames]{color}
\usepackage[colorlinks=true,linkcolor=Red,citecolor=Green]{hyperref}
 
\usepackage{tikz}
\DeclareMathAlphabet{\pazocal}{OMS}{zplm}{m}{n}

\usepackage{graphicx}
\usepackage[colorlinks=true]{hyperref}
\usepackage{color}
\usepackage{amsmath,amsfonts}
\usepackage{amsmath,environ}

\numberwithin{equation}{section}
\usepackage{amsmath, amsthm}
    \usepackage{dsfont}
    \usepackage{fancybox}
    \usepackage{amssymb}

\newcommand{\R}{\mathbb{R}}

\newcommand{\p}{{\partial}}
\newcommand{\olambda}{{\overline{\lambda}}}

\newcommand{\dd}[2]{\dfrac{\partial #1}{\partial #2}}

\newcommand{\Det}{{\text{Det}}}
\newcommand{\PP}{{\mathcal{P}}}
\newcommand{\DD}{{\mathcal{D}}}

\newcommand{\epsi}{\varepsilon}

\newcommand{\te}{\theta}
\newcommand{\lr}[1]{\langle #1 \rangle}
\newcommand{\FF}{\mathcal{F}}

\newcommand{\supp}{\mathrm{supp}}
\newcommand{\trho}{\tilde{\rho}}

\newcommand{\tB}{\widetilde{B}}
\newcommand{\oT}{{\overline{T}}}

\newcommand{\Dd}{\mathbb{D}}

\newcommand{\TT}{\mathcal{T}}
\newcommand{\Ss}{\mathbb{S}}
\newcommand{\Id}{{\operatorname{Id}}}

\newcommand{\Ell}{{\operatorname{Ell}}}

\newcommand{\VV}{{\mathcal{V}}}

\newcommand{\EE}{\mathcal{E}}
\newcommand{\WF}{{\operatorname{WF}}}

\newcommand{\MM}{\mathcal{M}}
\newcommand{\UU}{{\mathcal{U}}}
\newcommand{\LL}{{\mathcal{L}}}
\newcommand{\HH}{\mathcal{H}}

\newcommand{\SSS}{{\mathcal{S}}}

\newcommand{\systeme}[1]{\left\{ \begin{matrix} #1 \end{matrix} \right.}
\newcommand{\arXiv}[1]{\href{https://arxiv.org/abs/#1}{arXiv:#1}}

\newcommand{\tH}{{\widetilde{H}}}
\newcommand{\dive}{{\operatorname{div}}}

\newcommand{\C}{\mathbb{C}}

\newcommand{\az}{\alpha}

\newcommand{\tchi}{\tilde{\chi}}

\newcommand{\Trace}{{\operatorname{Tr}}}
\newcommand{\Op}{{\operatorname{Op}}}

\renewcommand{\Re}{\operatorname{Re}}
\renewcommand{\Im}{\operatorname{Im}}
\newcommand{\Mm}{\mathbb{M}}
\newcommand{\Res}{\mathrm{Res}}

\newcommand{\tP}{\widetilde{P}}

\newcommand{\tLambda}{{\widetilde{\Lambda}}}

\newcommand{\Range}{{\mathrm{Range}}}

\newcommand{\AAA}{\mathcal{A}}

\newcommand{\de}{\mathrel{\stackrel{\makebox[0pt]{\mbox{\normalfont\tiny def}}}{=}}}

\title[Pollicott--Ruelle resonances via kinetic Brownian motion]{Pollicott--Ruelle resonances via kinetic Brownian motion}
\author{Alexis Drouot}
\date{\today}

\NewEnviron{equations}{%
\begin{equation}\begin{gathered}
  \BODY
\end{gathered}\end{equation}
}

\NewEnviron{equations*}{%
\begin{equation*}\begin{gathered}
  \BODY
\end{gathered}\end{equation*}
}

\newtheorem{thm}{Theorem}

\newtheorem{rmk}{Remark}[section]
\newtheorem{lem}{Lemma}[section]
\newtheorem{proposition}[lem]{Proposition}

\newtheorem{theorem}[thm]{Theorem}
\theoremstyle{definition}

\begin{document}
\maketitle

\begin{abstract} The kinetic Brownian motion on the cosphere bundle of a Riemannian manifold $\Mm$ is a stochastic process that models the geodesic equation perturbed by a random white force of size $\epsi$. When $\Mm$ is compact and negatively curved, we show that the $L^2$-spectrum of the infinitesimal generator of this process converges to the Pollicott--Ruelle resonances of $\Mm$ as $\epsi$ goes to $0$. \end{abstract}

\section{Introduction}\label{sec:1}

We consider a smooth compact Riemannian manifold $\Mm$ with negative sectional curvatures and cosphere bundle $S^*\Mm$. The generator of the geodesic flow $H_1 \in TS^*\Mm$ has the Anosov property and, on suitable spaces, $P_0 \de \frac{1}{i} H_1$ has a discrete spectrum with eigenvalues called Pollicott--Ruelle resonances. We denote it by $\Res(P_0)$ These complex numbers appear in expansions of classical correlations -- see Tsuji \cite{Ts} and Nonnemacher--Zworski \cite{NZ}. We refer to \S \ref{sec:2.3} for precise definitions. 

Recently, several authors studied a stochastic process on $S^*\Mm$ called the kinetic Brownian motion -- see Franchi--Le Jan \cite{LF}, Grothaus--Stilgenbauer \cite{GS2}, Angst--Bailleul--Tardif \cite{ABT} and Li \cite{Li}. In contrast with the Langevin process, the kinetic Brownian motion models diffusive phenomena with finite speed of propagation. Its infinitesimal generator is $i P_\epsi \de H_1+\epsi \Delta_\Ss$, where $\Delta_\Ss \geq 0$ is the vertical spherical Laplacian -- see \S \ref{sec:2.1}.

In this paper, we investigate the convergence of the $L^2$-spectrum $\Sigma(P_\epsi)$ of $P_\epsi$, as $\epsi$ goes to $0^+$. Although the $L^2$-spectrum of $P_0$ is absolutely continuous and equal to $\R$, we have:

\begin{theorem}\label{thm:0} The set of accumulation points of $\Sigma(P_\epsi)$ as $\epsi \rightarrow 0^+$ is equal to $\Res(P_0)$.
\end{theorem} 

Theorem \ref{thm:6} below is a finer statement: the spectral projections of $P_\epsi$ depend smoothly on $\epsi$; and if each Pollicott--Ruelle resonance of $P_0$ is simple, the $L^2$-eigenvalues of $P_\epsi$ admit a full expansion in powers of $\epsi$. Remark \ref{rem:1} analyzes the convergence as $\epsi \rightarrow 0^-$. 

We proved Theorem \ref{thm:0} when $\Mm$ is an orientable surface in an earlier version \cite{D2} of this paper.

\subsection*{Motivation and outline of proof} Dyatlov--Zworski \cite{DZ2} showed that the Pollicott--Ruelle resonances of an Anosov vector field $X$ on a Riemannian manifold are the limits as $\epsi \rightarrow 0^+$ of the $L^2$-eigenvalues of $\frac{1}{i}(X + \epsi \Delta)$. From the point of view of partial differential equations, this realizes resonances as viscosity limits. From the point of view of probability theory, this indicates stochastic stability of Pollicott--Ruelle resonances, because the operator $\frac{1}{i}X + i \epsi \Delta$ generates the stochastic differential equation
\begin{equation}\label{Eq:0a}
\p_t\Phi_t = - X(\Phi_t) - \sqrt{2\epsi} B(t), \ \ \ \   \Phi_0 = \Id_\MM,
\end{equation}
where $B(t)$ is a Brownian motion on $\MM$. Their approach also shows that the $L^2$-eigenvalues of $\frac{1}{i}X + i \epsi \Delta$ converge to complex conjugates of Pollicott--Ruelle resonances as $\epsi \rightarrow 0^-$. This fact also holds here, see Remark \ref{rem:1}.

The geodesic flow on the cosphere bundle $S^*\Mm$ of a Riemannian manifold $\Mm$ is a fundamental example of Anosov flow. If $X$ denotes the generator of the geodesic flow, \eqref{Eq:0a} is a random perturbation of the geodesic equation. The perturbative term in \eqref{Eq:0a} acts on both momenta and positions. As was first modeled by Langevin's equation \cite{L08}, a physical random perturbation created by collisions should only act on the momentum variables. A generalization of Langevin's equation to cotangent bundles $T^*\Mm$ was studied in J$\o$rgensen \cite{J}, Soloveitchik \cite{So} and Kolokoltsov \cite{K}. 

In this paper, we remain on the cosphere bundle $S^*\Mm$ and we consider the \textit{kinetic Brownian motion}. This stochastic process is a random perturbation in the momentum random of the geodesic equation on $S^*\Mm$. It models diffusions with constant speed of propagation, and has generator $H_1 + \epsi \Delta_\Ss$. The kinetic Brownian motion was first introduced in Franchi--Le Jan \cite{LF}, as an extension of Langevin's equation in general relativity: it models the relativistic motion of random particles, whose speed has to be bounded by the speed of light. Grothaus--Stilgenbauer \cite{GS2} extended the construction to cosphere bundles of Riemannian manifolds, with applications to industry. Li \cite{Li} showed the first perturbative results in the small-and-large white force limit (respectively, $\epsi \rightarrow 0$ and $\epsi \rightarrow \infty$). Angst--Bailleul--Tardif \cite{ABT} improved upon Li's result and derived asymptotic in the context of rotationally invariant manifolds. We refer to \S\ref{sec:2.1} for precise definitions.

Dolgopyat--Liverani \cite{DL} studied another perturbation of the geodesic equation. They considered the geodesic motion of particles, coupled with an interaction of size $\epsi$. When the initial data is random and $\epsi$ goes to $0$, they showed that a suitable rescaling of the energy at time $t$ solves an explicit stochastic differential equation. Bernadin et al. \cite{BeL} obtained a formal expansion of the heat conductivity for systems of weakly coupled random particles. Conceptually, both results can be seen as a step towards deriving macroscopic equations from principles of microscopic dynamics. 

This paper aims to generalize the main result of Dyatlov--Zworski \cite{DZ2} to the kinetic Brownian motion. In contrast with \cite{DZ2}, the operator $P_\epsi = \frac{1}{i}(H_1+\epsi \Delta_\Ss)$ is hypoelliptic instead of being elliptic. An earlier version \cite{D2} contains a proof of Theorem \ref{thm:0} when $\Mm$ is an orientable surface. It can be seen as an introduction to the present paper. The technical details are simpler there, because in that case $\Delta_\Ss = -V^2$, with $V$ the generator of the circle action on the fibers of $S^*\Mm$.

The lack of ellipticity of $P_\epsi$ creates serious new difficulties that we overcome by showing that the operator $P_\epsi$ is maximally hypoelliptic in the regime $\epsi \rightarrow 0$, see Theorem \ref{thm:1}. For technical reasons, we will lift $P_\epsi$ to an operator $\tP_\epsi$ acting on functions on the orthonormal coframe bundle of $\Mm$. The proof continues with the Lebeau \cite{L}, where the maximal hypoellipticity of Bismut's hypoelliptic Laplacian \cite{B} is shown. Lebeau ingeniously uses certain commutation relations to reduce his study to the case of the model operator $x_1^2 D_{x'}^2 + D_{x_1}$, microlocally near $(0, x', 0, \xi')$, $\xi' \neq 0$. In our approach, we bypass the microlocal reduction and we work directly with $P_\epsi$. We replace Lebeau's main step with a positive commutator argument. This yields a maximal hypoellipticity result for $\tP_\epsi$, that descends to an estimate for $P_\epsi$. Lifting geometric equations to the orthonormal frame bundle has been an efficient technique in probability theory, starting with the pioneering constructions of stochastic processes on manifolds by Elworthy \cite{El}. It was used in both Li \cite{Li} and Angst--Bailleul--Tardif \cite{ABT} to show asymptotic results for the kinetic Brownian motion.

The remainder of the proof of Theorem \ref{thm:0} is similar to \cite{DZ2}.  We will decompose the operator $P_\epsi$ in two parts $P^\sharp_\epsi + P^\flat_\epsi$. The first part acts on momentum frequencies greater than $\epsi^{-1}$, and the maximal hypoelliptic estimate will take care of it. For the second part,  we will use the anisotropic Sobolev spaces designed in Faure--Sj\"ostrand \cite{FS} in a modified form due to Dyatlov--Zworski \cite{DZ1}. Their construction relies on Melrose's propagation estimate at radial points \cite{Me}, in the improved version of \cite[Propositions 2.6-2.7]{DZ1}. For the original  version of anisotropic spaces used in Anosov dynamics, see Baladi \cite{Ba}, Liverani \cite{Liv}, Gou\"ezel--Liverani \cite{GoLi} and Baladi--Tsuji \cite{BaTs}. We also mention Vasy \cite{Va} for application of similar anisotropic Sobolev spaces in the context of asymptotically hyperbolic manifolds and general relativity.

The operator $P_\epsi$ can be realized as the restriction of the hypoelliptic Laplacian of Bismut \cite{B} to the cosphere bundle. This connection provides another motivation for the study of $P_\epsi$. Li \cite{Li} and Angst--Bailleul--Tardif \cite{ABT} showed that the kinetic Brownian motion interpolates between geodesic trajectories as $\epsi \rightarrow 0$ and the Brownian motion on $\Mm$ as $\epsi \rightarrow \infty$ (after projection and rescaling). This dramatically echoes Bismut--Lebeau's motivation to study the hypoelliptic Laplacian, obtained in \cite{BL1} as an operator interpolating between the generator of the geodesic flow and the Laplacian on $\Mm$ (after rescaling and projection). For the corresponding interpretation in probability theory, see Bismut \cite{B2}. Improving upon work of Bismut \cite{BL2}, Shen \cite{S} recently obtained far-reaching applications of the hypoelliptic Laplacian, including a proof of Fried's conjecture \cite{Fr} for maximally symmetric spaces.

Baudoin--Tardif \cite{BT} showed exponential convergence of the heat operator $e^{-itP_\epsi}$ to equilibrium: there exists $\nu_\epsi > 0$ such that for every $u \in S^\infty(S^*\Mm)$, 
\begin{equation*}
\left|e^{-itP_\epsi} u-\int_{S^*\Mm} u\right| \leq Ce^{-\nu_\epsi t}\left|u-\int_{S^*\Mm} u\right|.
\end{equation*}
Because of the connection of $P_\epsi$ with the Laplacian on $\Mm$, Baudoin and Tardif expected that the optimal value of $\nu_\epsi$ converges as $\epsi \rightarrow \infty$ to the first eigenvalue of the \textit{non-negative} Laplacian on $\Mm$. Though the explicit value of $\nu_\epsi$ derived there converges to $0$ as $\epsi \rightarrow \infty$. When $\Mm$ is negatively curved, we conjecture that the optimal value of $\nu_\epsi$ converges as $\epsi \rightarrow 0$ to the largest imaginary parts of Pollicott--Ruelle~resonances~of~$\frac{1}{i}H_1$.

When $\Mm$ is not negatively curved, we can still study the accumulation points of the $L^2$-eigenvalues of $P_\epsi$ as $\epsi \rightarrow 0$.  Already in the case of the $2$-torus, the behavior of this spectrum is quite mysterious. See (in a slightly different context) \cite[Figure 3]{DZ2} and the discussion following it, originating from Galtsev--Shafarevich \cite{GS06}. The general case is far from being understood. Recently, Dyatlov--Zworski \cite{DZ3} showed a deep connection between Pollicott--Ruelle resonances and topology: the order of vanishing of the Ruelle zeta function at $0$ determines the genus of a negatively curved surface. We believe that the spectrum of $P_\epsi$ relates closed geodesics and topology, even when $\Mm$ is not negatively curved. The maximal hypoelliptic estimate \eqref{eq:5k} holds with no restrictions on the sign of the curvature. However, the methods of \S\ref{sec:5} are strictly restricted to the negative curvature case.

\noindent \textbf{Acknowledgment.} We are very grateful to Maciej Zworski for suggesting the problem and for his invaluable guidance. We would also like to thank Semyon Dyatlov and Gilles Lebeau for various fruitful discussions. This research was partially supported by the Fondation CFM and the National Science Foundation grant DMS-1500852.

\section{Preliminaries}

\subsection{The generator of the kinetic Brownian motion}\label{sec:2.1} Let $\Mm$ be a smooth compact Riemannian manifold of dimension $d \geq 2$ and $S^*\Mm = \{(z,\zeta^1) \in T^* \Mm : |\zeta^1|=1\}$ be its cosphere bundle -- the notation $\zeta^1$ instead of $\zeta$ for the dual variable of $z$ will be clear later. The restriction of the Liouville $1$-form $z \cdot d\zeta^1$ to $S^*\Mm$ is a contact form on $S^*\Mm$, denoted $\az$. Its Reeb vector field $H_1$ generates the geodesic flow. In particular, $\az(H_1) = 1$, $d\az(H_1,\cdot) = 0$, and $H_1$ is divergence-free with respect to the Liouville measure $\mu = \az \wedge (d\az)^{d-1}$. The $L^2$-norm with respect to this measure will be denoted by $|\cdot|$. Alternatively, $H_1$ is the restriction to $T^*\Mm$ of the Hamiltonian vector field of the function $\frac{1}{2}|\zeta^1|^2$, with respect to the canonical symplectic form $dz \wedge d\zeta^1$ on $T^*\Mm$.

For every $z \in \Mm$, the fiber $T_z^*\Mm$ admits a Euclidean structure and the fiber $S^*_z\Mm$, provided with the induced metric, is a Riemannian submanifold of $T_z^*\Mm$. The \textit{non-negative} Laplacian on $S_z^*\Mm$ is a differential operator $\Delta_\Ss(z) : C^\infty(S^*_z\Mm) \rightarrow C^\infty(S^*_z\Mm)$. Varying $z$ we obtain a differential operator $\Delta_\Ss : C^\infty(S^*_z\Mm) \rightarrow C^\infty(S^*_z\Mm)$ called the spherical vertical Laplacian. Similarly there is a spherical vertical gradient operator $\nabla_\Ss : C^\infty(S^*\Mm) \rightarrow C^\infty(TS^*\Mm)$, defined on each fiber $S_z^*\Mm$ as the standard gradient.

Let $P_\epsi$ be the operator
\begin{equation*}
P_\epsi \de \dfrac{1}{i}(H_1 + \epsi \Delta_\Ss) = \dfrac{1}{i}H_1 - i\epsi \Delta_\Ss,
\end{equation*}
with $L^2$-domain $D(P_\epsi) \de \{ u \in L^2 : \ P_\epsi u \in L^2\}$ -- here $P_\epsi u$ is seen as a distribution. Angst--Bailleul--Tardif \cite{ABT} call $P_\epsi$ the generator of the \textit{kinetic Brownian motion}. In \S \ref{sec:2.2.4} below we compute certain Lie brackets, showing that $P_\epsi$ satisfies H\"ormander's condition \cite{Ho} for hypoellipticity. The Rothschild--Stein theory of hypoelliptic operators \cite[\S 18]{RS} yields the subelliptic estimate \eqref{eq:0pa}: there exists a constant $c_\epsi > 0$ such that $|u|_{H^{2/3}} \leq c_\epsi(|P_\epsi u|+|u|)$. A significant part of this paper, \S \ref{sec:3}, studies the behavior of $c_\epsi$ as $\epsi \rightarrow 0$ when $H^{2/3}$ is replaced by its semiclassical version $H^{2/3}_\epsi$.

The operator $P_\epsi$ is semibounded: $\Re(\lr{iP_\epsi u,u}) \geq 0$. Combined with the hypoellipticity of $P_\epsi$ and the compactness of $S^*\Mm$, this shows that $P_\epsi$ has a discrete spectrum on $L^2$. This paper studies the accumulation points as $\epsi \rightarrow 0$ of the $L^2$-eigenvalues of $P_\epsi$ when $\Mm$ has negative curvature.

\subsection{Operators on frame bundles} 

This section reviews Cartan's lifting process from the cosphere bundle $S^*\Mm$ to the bundle of orthonormal frames $O^*\Mm$. Angst--Bailleul--Tardif \cite{ABT} and Li \cite{Li} previously used it to show asymptotic of the kinetic Brownian motion in the limits $\epsi \rightarrow 0, \infty$. We mention that when $\Mm$ is an orientable surface, $O^*\Mm \equiv S^*\Mm \times \{\pm 1\}$ and this lifting process is unnecessary. This simplifies the technical aspects in the earlier version \cite{D2} of this paper.

\subsubsection{Horizontal and vertical vector fields.} The space of frames at $z \in \Mm$ -- denoted $\FF_z^*\Mm$ -- is the vector space of linear maps $\zeta : \R^d \rightarrow T_z^*\Mm$. At this point $\zeta$ is not required to be orthogonal nor an invertible. The space $\FF_z^*\Mm$ is a Euclidean when provided with the scalar product $(\zeta,\zeta') \mapsto \Trace(\zeta^* \zeta')$. Varying the base point $z$ we obtain a vector bundle $\FF^*\Mm$ over $\Mm$ which admits a Riemannian structure.

For $(z_0,\zeta_0) \in \FF^*\Mm$, a vector $X_0 \in T_{z_0,\zeta_0}\FF^*\Mm$ is said to be vertical if $X_0$ is tangent to the fiber $\FF_{z_0}^*\Mm$. A smooth vector field $X \in T\FF^*\Mm$ is vertical if $X(z_0,\zeta_0)$ is vertical for all $(z_0,\zeta_0) \in \FF^*\Mm$. A curve $t \mapsto (z_t,\zeta_t) \in \FF^*\Mm$ is said to be horizontal if for all $e \in \R^d$, $\zeta_t(e)$ (which belongs to $T_{z_t}\Mm$) is parallel along $z_t$ with respect to the Levi--Civita connection. A vector $X_0 \in T_{z_0,\zeta_0}\FF^*\Mm$ is horizontal if there exists a horizontal curve $(z_t,\zeta_t)$ with $\p_t (z_t,\zeta_t)(0) = X_0$; a smooth vector field $X \in T\FF^*\Mm$ is horizontal if $X(z,\zeta)$ is horizontal for every $(z,\zeta) \in \FF^*\Mm$. 

The bundle of orthonormal frames $O^*\Mm$ is the subbundle of $\FF^*\Mm$ with fibers formed of orthogonal maps $\zeta : \R^d \rightarrow T_z^*\Mm$. Since parallel transport preserves angles, the Levi--Civita derivative of an orthogonal frame along a curve  is still an orthogonal frame. Vertical and horizontal vector fields in $TO^*\Mm$ are defined similarly as before. We also observe that $O^*\Mm$ is a bundle over $S^*\Mm$, provided with the projection $\pi_\Ss : (z,\zeta) \mapsto (z,\zeta(e_1))$, where $e_1=(1,0,...,0) \in \R^d$.

Geodesics on $\Mm$ are identified with integral curves of the vector field $H_1$ defined in \ref{sec:2.1}; the geodesic flow is then $\exp(tH_1)$. The vector field $H_1$ lifts to a horizontal vector field $\tH_1$ on $O^*\Mm$ defined as follows. Fix $(z_0,\zeta_0) \in O^*\Mm$ and let $(z_0,\zeta_0^1) = (z_0,\zeta_0(e_1))$ be its projection on $S^*\Mm$; let $(z_t,\zeta_t^1) = \exp(tH_1)(z_0,\zeta_0(e_1))$ be the geodesic starting at $(z_0,\zeta_0^1)$. Parallel transport of $\zeta_0$ along $z_t$ yields a flow $(z_t,\zeta_t) = \Phi_t(z_0,\zeta_0)$ on $\FF^*\Mm$. Since the parallel transport preserves angles this flow actually takes values in $O^*\Mm$. As $(z_t,\zeta_t^1)$ is a geodesic, $\zeta_t^1 = \zeta_t(e_1)$ is the parallel transport of $\zeta_0^1$ along $z_t$ hence $\zeta_t(e_1) = \zeta_t^1$. This shows that $(z_t,\zeta_t)$ is a lift of $(z_t,\zeta_t^1)$ to the orthogonal frame bundle. The vector field $\tH_1 \in T O^*\Mm$ is the generator of $\Phi_t$:
\begin{equation*}
\tH_1(z_0,\zeta_0) \de \left.\dfrac{d}{dt}\right|_{t=0} \Phi_t(z_0,\zeta_0).
\end{equation*} 
The integral curves of $\tH_1$ are horizontal, which shows that $\tH_1$ is horizontal.

Let $E_{k\ell}$ be the matrix $E_{k\ell} \de (\delta_{ki} \delta_{j\ell})_{ij}$ and $A_{k \ell}$ be the anti-symmetric matrix $A_{k\ell} \de E_{k \ell}-E_{\ell k}$. The matrix $e^{tA_{k\ell}}$ is orthogonal and $V_{k \ell}$ is the vector field on $O^*\Mm$ given by
\begin{equation*}
V_{k \ell}(z,\zeta) = \left. \dfrac{d}{dt} \right|_{t=0} \left(z,\zeta \circ e^{tA_{k \ell}}\right).
\end{equation*}
Since the projection of $(z,\zeta \circ e^{tA_{k\ell}})$ on $\Mm$ does not depend on $t$ the vector fiels $V_{k\ell}$ are vertical. The brackets of $\tH_1$ with $V_{1k}$ define new vector fields on $O^*\Mm$: $\tH_k \de [\tH_1, V_{1k}]$.

\subsubsection{Expression in coordinates} A system of coordinates $z_m \in \R^d$ on $\Mm$ lifts canonically to a system of coordinates $(z_m, \zeta_j^1)$ on $T^*\Mm$. If $(z,\zeta) \in \FF^*\Mm$ then $\zeta(e_i) \in T_z^*\Mm$ and we denote by $\zeta^i_j$ its coordinates. This defines a system of coordinates on $\FF^*\Mm$. 

Unless precised otherwise, all the sums appearing below are run through indices from $1$ to $d$. Let $(z,\zeta) \in O^*\Mm \subset \FF^*\Mm$ with coordinates $(z_m,\zeta_i^j)$. Then
\begin{equation*}
\zeta \circ e^{tA_{k\ell}} (e_i) = \zeta  + t\zeta(A_{k\ell}e_i) + O(t^2) = \zeta  + t\delta_{i\ell} \zeta(e_k)-t\delta_{ik}\zeta(e_\ell) + O(t^2).
\end{equation*}
%(Here the error term is measured with respect to the metric on $\FF^*\Mm$.)
Hence $\zeta \circ e^{tA_{k\ell}}$ has coordinates $\zeta^i_j + t \delta_{i\ell} \zeta^k_j -t \delta_{ik}\zeta^\ell_j + O(t^2)$ and
\begin{equation}\label{eq:0yab}
V_{k\ell} = \sum_{i,j} \left(\delta_{i\ell} \zeta^k_j - \delta_{ik}\zeta^\ell_j \right)\dd{}{\zeta^i_j} = \sum_j \zeta^k_j \dd{}{\zeta^\ell_j} - \zeta_j^\ell \dd{}{\zeta^k_j}.
\end{equation}

Geodesic trajectories $(z,\zeta^1) \in T^*\Mm$ satisfy the equation
\begin{equation*}
\dot{z}_m = \zeta_m^1, \ \ \dot{\zeta}_m^1 = \sum_{i,j} \Gamma_{ij}^m(z) \zeta_i^1 \zeta_j^1
\end{equation*}
while covectors $\eta \in T^*\Mm$ that are parallely transported along $(z,\zeta^1)$ satisfy
\begin{equation*}
\dot{\eta}_m = - \sum_{i,j} \Gamma_{ij}^m \zeta_j^1 \eta_i.
\end{equation*}
This yields the coordinate expression of $\tH_1$, $\tH_m$:
\begin{equation*}
\tH_1 = \sum_i \zeta_i^1 \dd{}{z_i} - \sum_{i,j,k,\ell} \Gamma_{ij}^\ell \zeta^1_i\zeta_j^k  \dd{}{\zeta_\ell^k},  \ \ \ \ 
\tH_m = \sum_i \zeta_i^m \dd{}{z_i} - \sum_{i,j,k,\ell} \Gamma_{ij}^\ell \zeta^m_i\zeta_j^k  \dd{}{\zeta^k_\ell}.
\end{equation*}

\subsubsection{Some differential operators}\label{sec:2.2.3}
Recall that $\Delta_\Ss$ is the operator defined in \S \ref{sec:2.1} and let $\Delta_\Mm$ the \textit{non-negative} Laplacian operator of $\Mm$. The operator $\Delta \de \Delta_\Mm+\Delta_\Ss$ is an elliptic operator acting on $C^\infty(S^*\Mm)$. 

The operators $\Delta_O^V, \Delta_O^H$ acting on $C^\infty(O^*\Mm)$ are defined by $\Delta_O^V \de -\sum_{i, j} V_{ij}^2$, $\Delta_O^H \de -\sum_{i} \tH_i^2$. The operator $\Delta_O \de \Delta_O^H+\Delta_O^V$ is an elliptic operator on $O^*\Mm$. Let $\pi_\Ss : (z,\zeta) \in O^*\Mm \mapsto (z,\zeta(e_1)) \in S^*\Mm$ be the bundle projection of $O^*\Mm$ to $S^*\Mm$. It lifts the operators $\Delta_O^V, \Delta_O^H, \tH_1$ as follows:
\begin{equation}\label{eq:1b}
\Delta_O^V \pi^*_\Ss = \pi^*_\Ss \Delta_\Ss, \ \ \ \ \Delta_O^H \pi^*_\Ss = \pi_\Ss^* \Delta_\Mm, \ \ \ \ \pi_\Ss^* H_1 = \tH_1 \pi_\Ss^*.
\end{equation}

\begin{proof}[Proof of \eqref{eq:1b}] In order to prove the first identity of \eqref{eq:1b} it is enough to show that for every $z \in \Mm$, $\pi_\Ss(z)^* \Delta_\Ss(z) = -\pi_\Ss(z)^*\sum_{i,j} V_{ij}^2(z)$, where $\pi_\Ss(z)$ is the canonical projection $\zeta \in O_z^*\Mm \rightarrow \zeta(e_1) \in S_z^*\Mm$ and $V_{ij}(z) = V_{ij}|_{C^\infty(O_z^*\Mm)}$. 
%\tr{maybe avoid use of normal coordinates here}
Normal coordinates centered at $z$ on $\Mm$ induce coordinates $\zeta_i^1$ on $T_z^*\Mm$ (and $\zeta_i^j$ on $\FF_z^*\Mm$). In these coordinates the Euclidean metric on $T_z^*\Mm$ takes the form $\sum_i (d\zeta_i^1)^2$; hence they provide an isometric identification of $S_z^*\Mm$ with $\Ss^{d-1}$, $O_z^*\Mm$ with $O(d)$, and $\FF_z^*\Mm$ with $\R^{d \times d}$. Therefore, it suffices to show that if $\pi_{\Ss^{d-1}} : O(d) \rightarrow \Ss^{d-1}$ is the canonical projection, if $\Delta_{\Ss^{d-1}}$ and $\Delta_{O(d)}$ are respectively the Laplacians on $\Ss^{d-1}$ and $O(d)$ (with respect to the metric induce by the Euclidean structure of $\R^{d\times d}$), then
\begin{equation}\label{eq:5l}
 \Delta_{O(d)} \pi^*_{\Ss^{d-1}}= \pi^*_{\Ss^{d-1}}\Delta_{\Ss^{d-1}}.
\end{equation}
This identity should be available in the literature, though we have found no reference. We prove it below.

Since $\Ss^{d-1} \subset \R^d \subset \R^{d \times d}$, $\Delta_{\Ss^{d-1}}$ can be written as $-\sum_j X_j^2$, where the $X_j$ are the projections of $\p_{\zeta_j^1}$ on $\Ss^{d-1}$ -- see \cite[Theorem 3.1.4]{Hsu}. In coordinates,
\begin{equation}\label{eq:5m}
X_j = \p_{\zeta_j^1} - \sum_k \zeta_j^1 \zeta_j^k \p_{\zeta_j^k}.
\end{equation}
A direct computation combining \eqref{eq:5m} with $\sum_j (\zeta_j^1)^2 = 1$ on $\Ss^{d-1}$ shows that if $u$ is a function on $\R^{d \times d} $ depending only on $(\zeta^1_1, ..., \zeta^1_d)$,
\begin{equation*}
\Delta_{\Ss^{d-1}} u|_{\Ss^{d-1}}   = -\sum_j \dd{^2 u}{{\zeta_j^1}^2} + \sum_{j,k} \zeta_j^1 \zeta_k^1 \dd{^2u}{\zeta_k^1 \p\zeta_j^1}  
      + (d-1) \sum_k \zeta_k^1 \dd{u}{\zeta_k^1}.
\end{equation*}
We similarly compute $\Delta_{O(d)} u|_{O(d)}$. Using \eqref{eq:0yab} and that $u$ depends only on $(\zeta^1_1, ..., \zeta^1_d)$,
\begin{equations*}
\Delta_{O(d)}u|_{O(d)} =  - \sum_{k,\ell} \left( \sum_j \zeta^k_j \dd{}{\zeta^\ell_j} - \zeta_j^\ell \dd{}{\zeta^k_j} \right)^2 u =\sum_{i > 1} \sum_{j,k} \left(\zeta_j^1 \dd{}{\zeta_k^i} - \zeta_j^i \dd{}{\zeta^1_j}\right) \zeta_k^i \dd{u}{\zeta^1_k} \\
 = -\sum_{i > 1} \sum_{j,k} \zeta_j^i \zeta_k^i \dd{^2 u}{\zeta^1_k \p \zeta^1_j} +\sum_{i > 1} \sum_{j,k} \zeta_j^1 \delta_{jk} \dd{u}{\zeta^1_k}  =  -\sum_{i > 1} \sum_{j,k} \zeta_j^i \zeta_k^i \dd{^2u}{\zeta^1_k \p \zeta^1_j} +(d-1) \sum_j \zeta_j^1 \dd{u}{\zeta^1_j}.
\end{equations*}
Because of these formula, proving \eqref{eq:5l} amounts to show that for $\zeta \in O(d)$,
\begin{equation}\label{eq:0u}
\sum_j \dd{^2 u}{{\zeta_j^1}^2} =  \sum_{i,j,k} \zeta_j^i \zeta_k^i \dd{^2u}{\zeta^1_k \p \zeta^1_j}.
\end{equation}
Since $\zeta \in O(d)$, $\zeta^* \in O(d)$ which implies that $\sum_i \zeta_j^i \zeta_k^i = \delta_{jk}$. This relation shows that \eqref{eq:0u} holds on $O(d)$, which proves \eqref{eq:5l} and the first identity of \eqref{eq:1b}.

The second identity in \eqref{eq:1b} is \cite[Proposition 3.1.2]{Hsu}. 

If $(z_t,\zeta_t) = \exp(t\tH_1)(z_0,\zeta_0)$ with $(z_0,\zeta_0) \in O^*\Mm$ then $\pi_\Ss(z_t,\zeta_t)$ is the geodesic starting at $\pi_\Ss(z_0,\zeta_0)$:  $\pi_\Ss(z_t,\zeta_t) = \exp(tH_1) \pi_\Ss(z_0,\zeta_0)$. The identity $\pi_\Ss^* H_1 = \tH_1\pi_\Ss^*$ follows.
\end{proof}

We define $\tP_\epsi \de \frac{1}{i}(\tH_1 + \epsi \Delta_O^V)$. Because of \eqref{eq:1b}, the operator $\tP_\epsi$ is the lift of $P_\epsi$ to the orthogonal coframe bundle: $\tP_\epsi \pi_\Ss^* = \pi_\Ss^* P_\epsi$.

\subsubsection{Commutation identities}\label{sec:2.2.4}

A computation using \eqref{eq:0yab} yields the commutation relation
\begin{equation}\label{eq:1d}
[V_{k\ell},V_{mn}] = \delta_{\ell m}V_{k n} + \delta_{n k} V_{\ell m} + \delta_{km} V_{n\ell} + \delta_{\ell n} V_{mk}.
\end{equation} 

We next study the commutation relations between the $V_{k\ell}$ and $\tH_m$. Fix $z \in \Mm$ together with normal coordinates centered at $z$. In particular, $\Gamma_{ij}^\ell(z) = 0$ and
\begin{equation*}
[V_{k\ell},\tH_m](z) = \sum_{i,j} [\zeta^k_j \p_{\zeta_j^\ell} - \zeta^\ell_j \p_{\zeta^k_j} ,\zeta_i^m \p_{z_i}] = \sum_{i} \delta_{\ell m} \zeta^k_i  \p_{z_i} - \delta_{km} \zeta^\ell_i \p_{z_i} = \delta_{\ell m}\tH_k(z) - \delta_{km} \tH_\ell(z).
\end{equation*}
Since $z$ was arbitrary, this shows that 
\begin{equation}\label{eq:1e}
[V_{k\ell},\tH_m] = \delta_{\ell m}\tH_k - \delta_{km} \tH_\ell.
\end{equation}

We conclude this section by proving that the operators $\Delta_O^V, \Delta_O^H$ enjoy some important commutation properties:
\begin{equation}\label{eq:1c}
[\Delta_O^V, V_{mn}] = 0, \ \ \ \  [\Delta_O^H, \Delta_O^V] = 0, \ \ \ \  [\Delta_\Mm, \Delta_\Ss] = 0.
\end{equation}

\begin{proof}[Proof of \eqref{eq:1c}] We start with the first identity. By \eqref{eq:1d}, $[\Delta_O^V, V_{mn}]$
\begin{equations*}
=-\sum_{k, \ell} \left(\delta_{\ell m} V_{k n} + \delta_{n k} V_{\ell m} + \delta_{km} V_{n\ell} + \delta_{\ell n} V_{mk}\right) V_{k\ell} +  V_{k\ell} \left(\delta_{\ell m} V_{k n} + \delta_{n k} V_{\ell m} + \delta_{km} V_{n\ell} + \delta_{\ell n} V_{mk}\right) \\
 = -\sum_k V_{kn} (V_{km}+V_{mk})  + (V_{km}+V_{mk}) V_{kn} + \sum_\ell V_{\ell n} (V_{\ell m}+V_{m\ell})  + (V_{\ell m}+V_{m\ell}) V_{k\ell} = 0,
\end{equations*}
where we used that $V_{ij} + V_{ji}=0$. 

For the second identity, we first observe that \eqref{eq:1e} implies
\begin{equations*}
[V_{k\ell},\Delta_O^H] = -\sum_m \left(\delta_{\ell m}\tH_k - \delta_{km} \tH_\ell\right) \tH_m + \tH_m \left(\delta_{\ell m}\tH_k - \delta_{km} \tH_\ell\right) \\
   = -\tH_k \tH_\ell + \tH_\ell \tH_k - \tH_\ell \tH_k + \tH_k \tH_\ell = 0.
\end{equations*}
Therefore $\Delta_O^H$ commutes with the $V_{k\ell}$ and a fortiori with $\Delta_O^V$.

The third identity is equivalent to $\pi_\Ss^*[\Delta_\Mm,\Delta_\Ss] = 0$. This is automatically satisfied since $[\Delta_O^H, \Delta_O^V] = 0$ and $\pi_\Ss^*$ intertwines $\Delta_\Mm$ with $\Delta_O^H$ and $\Delta_\Ss$ with $\Delta_O^V$ -- see \eqref{eq:1b}.\end{proof}

\subsubsection{Sobolev equivalence}\label{subsec:1} Recall that $\mu$ is the Liouville measure on $S^*\Mm$, that $\pi_\Ss$ denotes the bundle projection $O^*\Mm \rightarrow S^*\Mm$ and that $\pi_\Ss$ intertwines $\Delta_O$ with $\Delta$ -- see \eqref{eq:1b}. Let $\mu_O$ be a measure on $O^*\Mm$ with
\begin{equation}\label{eq:0o}
v \in  C^\infty(S^*\Mm) \ \Rightarrow \ \int_{S^*\Mm} v d\mu = \int_{O^*\Mm} \pi_\Ss^*v d\mu_O.
\end{equation}
Let $\Lambda_s = (\Id + \epsi^2 \Delta)^{s/2}$, $\tLambda_s = (\Id + \epsi^2 \Delta_O)^{s/2}$. We \textit{define} the semiclassical Sobolev space $H^s_\epsi$ on $S^*\Mm$ (resp. $\tH^s_\epsi$ on $O^*\Mm$) by $H^s_\epsi = \Lambda_{-s} L^2$ (resp. $\tLambda_{-s} L^2$) with the corresponding norm with respect to $\mu$ (resp. $\mu_O$). The identity \eqref{eq:0o} implies 
\begin{equation}\label{eq:7f}
|\pi_\Ss^*u|_{\tH^s_\epsi}^2 = \int_{O^*\Mm} \left|\tLambda_s \pi_\Ss^* u\right|^2 d\mu_O = \int_{S^*\Mm} \left|\Lambda _s u\right|^2 \mu = |u|_{H^s_\epsi}^2.
\end{equation}
The commutation relation \eqref{eq:1e} shows that the vector fields $V_{k\ell}, [V_{1m},\tH_1]$ span the whole tangent bundle $TO^*\Mm$. The operator $\tP_\epsi$ satisfies H\"ormander's condition \cite{Ho} for hypoellipicity, with only one commutator needed. The Rothschild--Stein theory \cite[\S 18]{RS} shows that there exists a constant $C_\epsi > 0$ such that
\begin{equation}\label{eq:0p}
v \in C^\infty(O^*\Mm) \ \Rightarrow \ 
|v|_{\tH^{2/3}_\epsi} \leq C_\epsi(|\tP_\epsi v| + |v|).
\end{equation} 
Thanks to \eqref{eq:7f}, this subelliptic estimate for $\tP_\epsi$ transfers to a subelliptic estimate on $P_\epsi$: it suffices to plug $v = \pi_\Ss^* u$ in \eqref{eq:0p} to obtain 
\begin{equation}\label{eq:0pa}
u \in C^\infty(S^*\Mm) \ \Rightarrow \ |u|_{H^{2/3}_\epsi} \leq C_\epsi(|P_\epsi u| + |u|).
\end{equation}

\subsubsection{Spherical vertical Laplacian as a sum of squares} We will need the following result: there exist $n > 0$ and $X_1, ..., X_n$ smooth vector fields on $S^*\Mm$ such that 
\begin{equation}\label{eq:7e}
\Delta_\Ss = -\sum_{j=1}^n X_j^2, \ \ \ \ \dive(X_j) = 0.
\end{equation}
Indeed, Nash's theorem shows there exist $n > 0$ and an isometric embedding $\iota : \Mm \hookrightarrow \R^n$. The manifold $S^*\Mm$ can be seen as a submanifold of $T^* \R^n$ thanks to the embedding 
\begin{equation*}
\left(z,\zeta^1\right) \mapsto \left(\iota(z), \ (d\iota(z)^*)^{-1} \cdot \zeta^1\right),
\end{equation*}
which in addition preserves the bundle structure. Let $X_1, ..., X_n$ be the orthogonal projections of $\p_{n+1}, ..., \p_{2n}$ on $S^*\Mm$. Following the proof of \cite[Theorem 3.1.4]{Hsu}, the $X_j$'s are divergence-free vector fields hence \eqref{eq:7e} holds.

\subsection{Dynamical systems and microlocal analysis}\label{sec:2.3}

The material here is mostly taken from \cite[\S 2.1]{DZ1} and \cite[Appendix E.5.2]{DZ}.

When $\Mm$ has negative curvature, $H_1$ generates an Anosov flow on $S^*\Mm$: there exists a decomposition of $TS^*\Mm$, invariant under the geodesic flow $e^{tH_1}$, of the form
\begin{equation*}
T_xS^*\Mm = E_0(x) \oplus E_u(x) \oplus E_s(x),
\end{equation*}
where $E_0(x) = \R \cdot H_1(x)$ and $E_u(x), E_s(x)$ satisfy:
\begin{equations*}
v \in E_u(x) \ \Rightarrow \ |de^{tH_1}(x) v| \leq Ce^{ct} |v|, \ \ t < 0, \\
v \in E_s(x) \ \Rightarrow \ |de^{tH_1}(x) v| \leq Ce^{-ct} |v|, \ \ t > 0.
\end{equations*}

For $(x,\xi) \in TS^*\Mm$, let $\sigma_{H_1}(x,\xi) = \lr{\xi,H_1(x)}$ -- a smooth function on $TS^*\Mm$. The Hamiltonian vector field $H_{\sigma_{H_1}}$ of $\sigma_{H_1}$ generates the flow $\exp(tH_{\sigma_{H_1}})$ given by
\begin{equation}\label{eq:9f}
\exp(tH_{\sigma_{H_1}})(x,\xi) = \left(e^{tH_1}(x), ^Tde^{tH_1}(x)^{-1}\xi\right).
\end{equation}
Since $\sigma_{H_1}(x,\xi) = \frac{1}{i}\lr{\xi,H_1(x)}$ is homogeneous of degree $1$ in $\xi$, $\exp(tH_{\sigma_{H_1}})$ extends to a map $\oT^*S^*\Mm \rightarrow \oT^*S^*\Mm$, see \cite[Proposition E.5]{DZ1}. A radial sink (with respect to $H_{\sigma_{H_1}}$) is a $\exp(tH_{\sigma_{H_1}})$-invariant closed conic set $L \subset T^* S^*\Mm \setminus 0$ with a conical neighborhood $U$ satisfying
\begin{equations}\label{Eq:7k}
t \rightarrow +\infty \ \Rightarrow \ \mathrm{d}(\kappa(\exp(tH_{\sigma_{H_1}})(U)), \kappa(L)) \rightarrow 0,  \\
(x,\xi) \in U \ \Rightarrow \ |\pi_\xi\exp(tH_{\sigma_{H_1}})(x,\xi)| \geq C^{-1} e^{c t} |\xi|.
\end{equations}
Here $\pi_\xi(x,\xi) = \xi$. A radial source is defined by reversing the flow direction in \eqref{Eq:7k}.

The decomposition $T_xS^*\Mm = E_u(x) \oplus E_0(x) \oplus E_s(x)$ induces a dual decomposition $T_x^* S^*\Mm = E_s^*(x) \oplus E_0^*(x) \oplus E_u^*(x)$. \textit{Note that in this notation, $E_s^*(x)$ is the dual of $E_u(x)$ and $E_u^*(x)$ is the dual of $E_s^*(x)$.} The stable and unstable foliations of Anosov flows are related to the radial source and sinks as follows: $E_s^* \setminus 0$ is a radial source and $E_u^*\setminus 0$ is a radial sink, see \cite[\S 2.3]{DZ1}.

Pollicott--Ruelle resonances are dynamical quantities associated to $\Mm$, that quantify the decay of classical correlations, see \cite[Corollary 1.2]{Ts} and \cite[Corollary 5]{NZ}. These numbers can also be realized as eigenvalues of $\frac{1}{i} H_1$ on specifically designed Sobolev spaces. They are the poles of the meromorphic continuation of the Fredholm family of operators $(P_0-\lambda)^{-1} = (\frac{1}{i}H_1-\lambda)^{-1} : C^\infty \rightarrow \DD'$, where $\DD'$ is the set of distributions on $S^*\Mm$. The poles of $(P_0-\lambda)^{-1}$ have finite rank; the multiplicity of a pole $\lambda_0 \in \C$ is $\text{rank}(\Pi_{\lambda_0})$, where
\begin{equation}\label{eq:5n}
\Pi_{\lambda_0} \de \dfrac{1}{2\pi i} \oint_{\p \Dd(\lambda_0,r_0)} (P_0-\lambda)^{-1} d\lambda
\end{equation}
and $r_0$ is small enough so that $\lambda_0$ is the unique pole of $(P_0-\lambda)^{-1}$ on $\Dd(\lambda_0,r_0)$. In order to investigate further the residues of $(P_0-\lambda)^{-1}$, we recall that one can associate to each $u \in \DD'$ a conical set $\WF(u)$, called the classical wavefront set, which measures in phase space where $u$ is not smooth. We refer to \cite[\S 7]{GS} for precise definitions. For $\Gamma \subset T^*S^*\Mm$ a conical set, let $\DD'_\Gamma$ be the set of distributions with classical wavefront set contained in $\Gamma$. 

\begin{lem}\label{lem:3} If $\lambda_0$ is a simple Pollicott--Ruelle resonance of $P_0 = \frac{1}{i}H_1$, there exist $u \in \DD'_{E_u^*}, v \in \DD'_{E_s^*}$ and a holomorphic family of operators $A(\lambda)$ defined near $\lambda_0$, with
\begin{equation*}
(P_0 - \lambda)^{-1} = \dfrac{u \otimes v}{\lambda-\lambda_0} + A(\lambda).
\end{equation*}
\end{lem}

\begin{proof} According to \cite[Proposition 3.3]{DZ1}, the operator $\Pi_{\lambda_0}$ defined in \eqref{eq:5n} is equal to $u \otimes v$, where $\WF(u) \subset E_u^*$, $\WF(v) \subset E_s^*$; and there exist $J > 0$ and a family of operators $A(\lambda) : C^\infty \rightarrow \DD'$ holomorphic near $\lambda_0$ such that
\begin{equation}\label{eq:0k}
(P_0 - \lambda)^{-1} = A(\lambda) + \sum_{j=1}^J \dfrac{(P_0-\lambda_0)^{j-1} \Pi_{\lambda_0}}{(\lambda-\lambda_0)^j}.
\end{equation}
By the same argument as in the proof of \cite[Theorem 2.4]{DZ} the operator $P_0 - \lambda_0$ maps $\Range(\Pi_{\lambda_0})$ to itself and $(P_0 - \lambda_0)|_{\Range(\Pi_{\lambda_0})}$ is nilpotent. Since $\Range(\Pi_{\lambda_0})$ has dimension $1$, $(P_0-\lambda)|_{\Range(\Pi_{\lambda_0})}$ is equal to $0$ and the index $J$ in \eqref{eq:0k} is equal to $1$.\end{proof}

In \cite{DZ1} the meromorphic continuation of $(P_0-\lambda)^{-1}$ is realized via analytic Fredholm theory. Therefore,  Pollicott--Ruelle resonances of $P_0$ are identified with the roots of a suitable Fredholm determinant, see \cite[Proposition 3.2]{DZ2}.

\subsection{Semiclassical analysis}\label{subsec:2.2} We recall some facts about the semiclassical calculus on $S^*\Mm$ (or $O^*\Mm$).  Unless specified otherwise, our basic reference is \cite[Appendix E]{DZ}. \textit{In the rest of the paper, $h$ is a parameter satisfying $0 < h < 1$.}

For $m \in \R$, a function $a \in C^\infty(T^*S^*\Mm)$ depending on $h$ lies in the symbol~class~$S^m$~if
\begin{equation*}
\forall \az, \beta, \  \exists C_{\az \beta} >0, \  \forall 0 < h < 1, \ \sup_{(x,\xi) \in T^* S^*\Mm} \lr{\xi}^{m-|\beta|} |\p_x^\az \p_\xi^\beta a(x,\xi)| \leq C_{\az \beta}.
\end{equation*}
Semiclassical pseudodifferential operators on $S^*\Mm$ are $h$-quantization of symbols in $S^m$ and form an algebra denoted $\Psi_h^m$, see \cite[Appendix E.1]{DZ}. Conversely, to each $A \in \Psi^m_h$, we can associate a principal symbol $\sigma(A) \in S^m/hS^{m-1}$. For example if $X$ is a smooth vector field on $S^*\Mm$, the principal symbols of $\frac{h}{i}X$ is 
 \begin{equation}\label{EQ:0a}
\sigma_X(x,\xi) \de \lr{\xi,X(x)} \mod hS^0.
 \end{equation}
The principal symbol of $h^2\Delta = h^2\Delta_\Ss + h^2 \Delta_O$ induces a positive definite quadratic form on the fibers of $T^*S^*\Mm$, thus a metric on $S^*\Mm$. We denote it by $g$, so that the principal symbol of $h^2 \Delta$ is $|\xi|^2_g$ modulo $hS^1$. We refer to \cite[Appendix E.1]{DZ} for additional properties of operators in $\Psi^m_h$.

Let $\oT^*S^*\Mm$ be the radial compactification of $T^*S^*\Mm$: it is a compact manifold with interior $T^*S^*\Mm$ and boundary $S^*S^*\Mm$ associated with a map $\kappa : T^*S^*\Mm \setminus 0 \rightarrow \p \oT^*S^*\Mm$, see \cite[Appendix E.1]{DZ}. To each operator $A \in \Psi^m_h$, we associate an invariant closed set $\WF_h(A) \subset \oT^* S^*\Mm$ called the wavefront set of $A$, which measures where $A$ is not semiclassically negligible. We also associate to $A$ an invariant open set $\Ell_h(A) \subset \oT^*S^*\Mm$ called the elliptic set, which measures where $A$ is semiclassically invertible. See \cite[Appendix E.2]{DZ} for precise definitions. The main interest of the elliptic set is the elliptic estimate \cite[Proposition 2.4]{DZ1}. Among classical results we record the sharp G$\mathring{\text{a}}$rding inequality \cite[Proposition E.34]{DZ} and the Duistermaat--H\"ormander propagation of singularities theorem~\cite[Proposition~2.5]{DZ1}. 

A less classical result needed here is the radial source (resp. radial sink) estimate, first introduced by Melrose %\tr{theorem} 
\cite{Me} and developed further in \cite[Propositions 2.6-2.7]{DZ1}. This estimate applies microlocally near a radial source (resp. near a radial sink); it enables us to control certain semiclassical quantities provided that the regularity index is high (resp. low) enough. This motivates the definition of semiclassical anisotropic Sobolev spaces that have high microlocal regularity near radial sources and low microlocal regularity near radial sinks. See \cite[Chapter 8]{Z} for a general theory of semiclassical anisotropic Sobolev spaces and \cite[\S 3.1]{DZ1} for the specific scale of Sobolev space we will use in this paper. 

We can also consider operators on $\R^n, n > 0$, that belong to a more general class than $\Psi_h^0$. These are realized as quantization of symbols $a$ satisfying
\begin{equation*}
\forall \az, \beta, \  \exists C_{\az \beta} >0, \forall 0 < h < 1 , \ \sup_{(x,\xi) \in T^* \R^n} |\p_x^\az \p_\xi^\beta a(x,\xi)| \leq C_{\az \beta}.
\end{equation*}
The space of resulting symbols (resp. resulting operators) is denoted by $S$ (resp. $\Psi_h$). This space is \textit{not} invariant under change of variables.  In the class $\Psi_h$, the remainders  in the composition formula are smaller than the leading part, but they are not mor smoothing -- in contrast with $\Psi_h^0$. We will use this class exclusively in \S \ref{subsec:3.3}. Our basic reference for such operators is \cite[Chapter 4]{Z}. 

\section{Maximal hypoelliptic estimates}\label{sec:3}

\subsection{Statement of the result}

Recall that the operator $P_\epsi$ is given by $\frac{1}{i}(H_1 + \epsi \Delta_\Ss)$, that the semiclassical Sobolev spaces $H^s_\epsi$ were defined in \S\ref{subsec:1}, and that there exist $X_1, ..., X_n \in T S^*\Mm$ such that such that $\Delta_\Ss = - \sum_{j=1}^n X_j^2$.  Here we prove an estimate for $P_\epsi$ similar to \cite[Theorem 18]{RS}, but uniform in the \textit{semiclassical regime} $\epsi \rightarrow 0$. Let $\rho_1, \rho_2$ be two smooth functions satisfying
\begin{equation}\label{eq:0g}
\supp(\rho_1, \rho_2) \subset \R \setminus 0, \ \ \ \ 1-\rho_1, \ 1-\rho_2 \in C^\infty_0(\R,[0,1]), \ \ \ \ \rho_2 = 1 \text{ on } \supp(\rho_1).
\end{equation}

\begin{theorem}\label{thm:1} Let $R > 0$ and $\rho_1, \rho_2$ two functions satisfying \eqref{eq:0g}. For any $N > 0$, there exists  $C_{N,R} > 0$ such that for every $|\lambda| \leq R$, $u \in C^\infty(S^*\Mm)$, and $0 < \epsi < 1$,
\begin{equations}\label{eq:5k}
\epsi^{2/3}|\rho_1(\epsi^2 \Delta) u|_{H^{2/3}_\epsi} + \epsi^{1/3} \sum_{j=1}^n| \epsi X_j \rho_1(\epsi ^2 \Delta) u|_{H^{1/3}_\epsi } + |\rho_1(\epsi ^2 \Delta) \epsi ^2 \Delta_\Ss u| \\ \leq C_{N,R} |\rho_2(\epsi ^2\Delta) \epsi (P_\epsi -\lambda)u| + O(\epsi ^N)|u|.
\end{equations}
\end{theorem}

This Theorem  applies to any smooth compact Riemannian manifold $\Mm$, \textit{with no restriction on the sign of its sectional curvatures, and with no change in the proof.}

The paper \cite{RS} shows that for every $\epsi  > 0$ there exists $C_\epsi  > 0$ such that
\begin{equation*}
|\rho_1(\epsi^2\Delta) u|_{H^{2/3}_\epsi } \leq C_\epsi (|\rho_2(\epsi^2\Delta)\epsi  (P_\epsi -\lambda)u| + |u|).
\end{equation*}
Theorem \ref{thm:1} shows that $C_\epsi  = O(\epsi ^{-2/3})$. Because of related estimates in \cite{DSZ} and \cite[\S 3]{L} we believe that this upper bound is optimal. This is the subject of a work in progress of Smith \cite{Sm}. 

We proved Theorem \ref{thm:1} in \cite{D2}, when $\Mm$ is an orientable surface. In this case, $\Delta_\Ss = -V^2$ where $V\in TS^*\Mm$ generates the circle action on the fibers of $S^*\Mm$. Thus, $\Delta_\Ss$ is a sum of squares of vector fields that commute with $\Delta_\Ss$, a fact used in a crucial manner in the proof of \cite[Proposition 3.1]{D2}. This no longer holds when $d \geq 3$ or $\Mm$ is not orientable.  In order to apply nevertheless the main idea of \cite{D2} we observe that $\Delta_O^V$ -- the lift of $\Delta_\Ss$ to the orthonormal coframe bundle $O^*\Mm$ -- is the sum of squares of vector fields which all commute with $\Delta_O^V$:
\begin{equation}\label{eq:0v}
\Delta_O^V = -\sum_{i,j} V_{ij}^2, \ \ \ \ [\Delta_O^V, V_{ij}] = 0,
\end{equation}
see \S\ref{sec:2.2.3}-\ref{sec:2.2.4}. The operator $P_\epsi = \frac{1}{i}(H_1+\epsi \Delta_\Ss)$ on $C^\infty(S^*\Mm)$ lifts to  $\tP_\epsi = \frac{1}{i}(\tH_1+\epsi \Delta_O^V)$ on $C^\infty(O^*\Mm)$. Because of \eqref{eq:0v}, we can modify the techniques of \cite{D2} to apply them to the operator $\tP_\epsi$. This will yield estimates for functions on $O^*\Mm$, which we will descend to function on $S^*\Mm$.

We will use semiclassical analysis to show Theorem \ref{thm:1}. To conform with standard notations, we define 
\begin{equation*}
h \de \epsi, \ \ \ \ P \de ihP_h = h^2 \Delta_\Ss + hH_1, \ \ \ \ \tP \de ih\tP_h = h^2 \Delta_O^V+h\tH_1,
\end{equation*}
\textit{for use in \S\ref{sec:3.2}-\ref{subsec:3.3} only}. We see $h$ as a small parameter and $P$ as a $h$-semiclassical operator in $\Psi_h^2$. As in \cite{D2}, we base our investigation on ideas of Lebeau \cite{L}, where a subelliptic estimate for the Bismutian is shown, for $\epsi = 1$.  The strategy starts to differ when Lebeau uses a microlocal reduction to a toy model. Instead, we continue to work with $P_\epsi$ and we replace the microlocal reduction by a positive commutator estimate. This avoids to use semiclassical Fourier integral operators.

\subsection{Reduction to a subelliptic estimate}\label{sec:3.2} The first lemma shows that Theorem \ref{thm:1} is a consequence of a subelliptic estimate.

\begin{lem}\label{lem:1a} Let $\SSS_1, ..., \SSS_q, \TT \subset \Psi_h^1$ be a collection of selfadjoint semiclassical operators on $S^*\Mm$ or $O^*\Mm$ and $\PP \de \sum_{j=1}^q \SSS_j^2 + i\TT$. There exist $C, h_0>0$ such that 
\begin{equation*}
0 < h < h_0 \ \Rightarrow \ |\TT u| +  \left|\sum_{j=1}^q  \SSS_j^2 u \right| + \sum_{j=1}^q h^{1/3} |\SSS_j u|_{H_h^{1/3}} \leq C |\PP u| + O(h^{2/3})|u|_{H^{2/3}_h}.
\end{equation*}
\end{lem}

\begin{proof} We prove the result only in the case of $S^*\Mm$; the proof is identical when considering operators on  $O^*\Mm$. We first show the estimate 
\begin{equation}\label{eq:1l} 
|\SSS_j u|_{H^{1/3}_h}^2 \leq C|\PP u| |u|_{H^{2/3}_h} + O(h) |u|_{H_h^{2/3}}^2.
\end{equation} 
Recall that $\Delta \de \Delta_\Mm + \Delta_\Ss$, where $\Delta_\Mm$ is the \textit{non-negative} standard Laplacian on $\Mm$ (lifted to $S^*\Mm$) and $\Delta_\Ss$ is the spherical Laplacian on $S^*\Mm$. The $H^s_h$-norm was defined in \S \ref{subsec:1} by $|u|_{H^s_h} \de |\Lambda_s u|$, where $\Lambda_s = (\Id+h^2 \Delta)^{s/2}$. Thus,
\begin{equation}\label{eq:7a} 
|\SSS_j u|_{H^{1/3}_h}^2 = |\Lambda_{1/3} \SSS_j u|^2 \leq 2|\SSS_j \Lambda_{1/3} u|^2 + 2|[\Lambda_{1/3},\SSS_j]u|^2 \leq 2|\SSS_j \Lambda_{1/3} u|^2 + O(h^2)|u|_{H^{1/3}_h}^2
\end{equation} 
because $[\Lambda_{1/3},\SSS_j] \in h \Psi^{1/3}_h$. Next we study $|\SSS_j \Lambda_{1/3} u|$: using $\sum_{j=1}^q \SSS_j^2 = \Re(\PP)$,
\begin{equations}\label{eq:7b} 
|\SSS_j \Lambda_{1/3} u|^2 = \lr{\SSS_j^2 \Lambda_{1/3} u, \Lambda_{1/3} u} \leq \lr{\sum_{j=1}^q \SSS_j^2 \Lambda_{1/3} u, \Lambda_{1/3} u} \\
\leq \Re(\lr{\PP \Lambda_{1/3} u, \Lambda_{1/3} u}) = \Re(\lr{\PP u, \Lambda_{2/3} u}) + \Re(\lr{[\PP,\Lambda_{1/3}]u,\Lambda_{1/3} u}).
\end{equations} 
We can estimate $\lr{\PP  u, \Lambda_{2/3} u}$ by $|\PP u| |u|_{H^{2/3}_h}$. The identity $\PP =\sum_{j=1}^q \SSS_j^2+i\TT$ yields
\begin{equation*} \Re(\lr{[\PP ,\Lambda_{1/3}]u,\Lambda_{1/3} u}) = \sum_{j=1}^q  \Re(\lr{[\SSS_j^2,\Lambda_{1/3}]u,\Lambda_{1/3} u}) + \Re(\lr{[i\TT, \Lambda_{1/3}] u, \Lambda_{1/3} u}). 
\end{equation*}  
The operator $[i\TT, \Lambda_{1/3}]$ belongs to $h\Psi^{1/3}_h$ therefore $|\lr{[i\TT, \Lambda_{1/3}] u, \Lambda_{1/3} u}| = O(h)|u|_{H^{1/3}_h}^2$. Using the relation $[\SSS_j^2,\Lambda_{1/3}] = \SSS_j [\SSS_j,\Lambda_{1/3}] + [\SSS_j,\Lambda_{1/3}] \SSS_j$ and the fact that $[\SSS_j,\Lambda_{1/3}]$ is anti-selfadjoint we obtain 
\begin{equations*} 
\lr{[\SSS_j^2,\Lambda_{1/3}]u,\Lambda_{1/3} u} = -\lr{ \SSS_j u,[\SSS_j,\Lambda_{1/3}]\Lambda_{1/3} u} +  \lr{[\SSS_j,\Lambda_{1/3}]u,\SSS_j\Lambda_{1/3} u} 
\\  = -\lr{ \SSS_j u,[\SSS_j,\Lambda_{1/3}]\Lambda_{1/3} u} + \lr{\Lambda_{1/3}[\SSS_j,\Lambda_{1/3}]u, \SSS_ju} + \lr{[\SSS_j,\Lambda_{1/3}] u, [\SSS_j,\Lambda_{1/3}] u}. 
\end{equations*} 
The operators $\Lambda_{1/3}[\SSS_j,\Lambda_{1/3}]$ and $[\SSS_j,\Lambda_{1/3}]$ belong to $h\Psi^{2/3}_h$ and $h\Psi^{1/3}_h$, respectively. Moreover $\SSS_j^2 \leq \Re(\PP)$, hence $|\SSS_ju| \leq |\PP u|^{1/2} |u|^{1/2}$. It follows that 
\begin{equations*} 
|\lr{[\SSS_j^2,\Lambda_{1/3}]u,\Lambda_{1/3} u}| \leq |\SSS_j u| |\Lambda_{1/3}[\SSS_j,\Lambda_{1/3}] u| + |[\SSS_j,\Lambda_{1/3}] u|^2 \\ \leq O(h)|\PP u|^{1/2} |u|^{1/2} |u|_{H^{2/3}_h} + O(h^2)|u|_{H^{1/3}_h}^2.
\end{equations*} 
Gluing this estimate with \eqref{eq:7a}, \eqref{eq:7b}, we get the bound  
\begin{equations*} 
|\SSS_ju|_{H_h^{1/3}}^2 \leq C|\PP u| |u|_{H^{2/3}_h} + O(h) |u|_{H_h^{1/3}}^2 + O(h) |\PP u|^{1/2} |u|^{1/2} |u|_{H^{2/3}_h} 
\\ \leq C|\PP u| |u|_{H^{2/3}_h} + O(h) |u|_{H_h^{2/3}}^2 + O(h) |\PP u|^{1/2} |u|_{H^{2/3}_h}^{3/2} \leq C|\PP u| |u|_{H^{2/3}_h} +O(h) |u|_{H_h^{2/3}}^2.
\end{equations*} 
In the last inequality we used $a b \leq a^2 + b^2$ with $a = |\PP u|^{1/2} |u|_{H^{2/3}_h}^{1/2}$ and $b = h |u|_{H^{2/3}_h}$. This proves \eqref{eq:1l}. We observe that \eqref{eq:1l} gives the estimate on $|\SSS_j u|_{H^{1/3}_h}$ provided by the lemma:
\begin{equation}\label{eq:7c}
h^{2/3} |\SSS_j u|^2 \leq C h^{2/3} |\PP u| |u|_{H_h^{2/3}} + O(h^{5/3})|u|_{H^{2/3}_h}^2 \leq C|Pu|^2 + O(h^{4/3})|u|_{H^{2/3}_h}^2.
\end{equation}

Next we observe that 
\begin{equation*} |\PP u|^2 = \left|\sum_{j=1}^q \SSS_j^2 u\right|^2+|\TT u|^2 + \sum_{j=1}^q \lr{ [\SSS_j^2,i\TT]u,u}. 
\end{equation*} 
To conclude the proof of the lemma it suffices to control the commutator terms $\lr{[\SSS_j^2,i\TT]u,u}$. We have 
\begin{equation*} 
\lr{[\SSS_j^2,i\TT]u,u} = \lr{[\SSS_j,i\TT]u,\SSS_ju}+ \lr{\SSS_ju,[\SSS_j,i\TT]u} = 2 \Re(\lr{\SSS_ju,[\SSS_j,i\TT]u}). 
\end{equation*}
By interpolation, $|\lr{[\SSS_j^2,i\TT]u,u}| \leq |\SSS_j u|_{H^{1/3}_h} |[\SSS_j,i\TT]u|_{H^{-1/3}_h}$. Since $[\SSS_j,i\TT] \in h\Psi^1_h$ it is bounded from $H^{2/3}_h$ to $H^{-1/3}_h$ with norm $O(h)$. By \eqref{eq:1l}, 
\begin{equation*}
 |\lr{[\SSS_j^2,i\TT]u,u}| \leq Ch \left( |\PP u|^{1/2} |u|_{H^{2/3}_h}^{1/2} + h^{1/2} |u|_{H_h^{2/3}} \right)  |u|_{H^{2/3}_h}.
\end{equation*} 
Hence we obtain 
\begin{equations}\label{eq:7d} 
\left|\sum_{j=1}^q  \SSS_j^2 u\right|^2+|\TT u|^2  \leq C |\PP u|^2 + O(h) |\PP u|^{1/2} |u|_{H^{2/3}_h}^{3/2} + O(h^{3/2}) |u|_{H^{2/3}_h}^2   \\ \leq C |\PP u|^2 + O(h^{4/3}) |u|_{H^{2/3}_h}^2.
\end{equations} 
In the second line we used Young's inequality: $ab \leq a^4 + b^{4/3}$ with $a = |\PP u|^{1/2}$, $b = h |u|_{H^{2/3}_h}^{3/2}$. The estimates \eqref{eq:7c}, \eqref{eq:7d} conclude the proof.\end{proof}

Roughly speaking, this lemma reduces the proof of \eqref{eq:0g} to an estimate of the form
\begin{equation}\label{eq:7g}
u \in C^\infty(S^*\Mm) \ \Rightarrow \ h^{2/3}|\rho_1(h^2 \Delta) u|_{H^{2/3}_h} \leq C |\rho_2(h^2 \Delta) P u| + O(h^\infty) |u|.
\end{equation}
Because of the reasons detailed above, we will work with the lift of $P$ to $O^*\Mm$ rather than directly with $P$. We will show the estimate
\begin{equation}\label{eq:7h}
v \in C^\infty(O^*\Mm) \ \Rightarrow \ h^{2/3}|\rho_1(h^2 \Delta_O) v|_{H^{2/3}_h} \leq C |\rho_2(h^2 \Delta_O) \tP v| + O(h^\infty) |v|.
\end{equation}
To see that \eqref{eq:7h} implies \eqref{eq:7g} we plug $v=\pi_\Ss^*u$ in \eqref{eq:7h}, then we use the identity \eqref{eq:1b} between $\tP$ and $P$, $\Delta$ and $\Delta_O$, and finally the relation \eqref{eq:7f} between Sobolev spaces on $S^*\Mm$ and $O^*\Mm$. The bound \eqref{eq:7h} will be implied by microlocal estimates on $\tP$:

\begin{proposition}\label{prop:1} For every $(x_0,\xi_0) \in \oT^*O^*\Mm \setminus 0$ there exists an open neighborhood $W_{x_0,\xi_0}$  of $(x_0,\xi_0)$ in $\oT^*O^*\Mm \setminus 0$ with the following property. For every $A \in \Psi_h^0$ with $\WF_h(A) \subset W_{x_0,\xi_0}$, there exists $B$ with $\WF_h(B) \subset W_{x_0,\xi_0}$ such that
\begin{equation*}
v \in C^\infty(O^*\Mm) \ \Rightarrow \ h^{2/3}|Av|_{\tH^{2/3}_h} \leq C|\tP Bv|+O(h)|v|_{\tH^{3/5}_h}.
\end{equation*}
\end{proposition}

\begin{proof}[Proof of Theorem \ref{thm:1} assuming Proposition \ref{prop:1}] It suffices to prove the Theorem when $h$ is sufficiently small. We first fix $N, R > 0$ and $\rho_1, \rho_2$ two functions satisfying \eqref{eq:0g}. Recall that we can write $P = -h^2 \sum_{j=1}^n X_j^2 + h H_1$, where $\frac{h}{i} X_j, \frac{h}{i} H_1$ are selfadjoint semiclassical operators in $\Psi_h^1$.

\textbf{Step 1.} By Lemma \ref{lem:1a} applied to $P$ instead of $\PP$ and $\rho_1(h^2\Delta) u$ instead of $u$,
\begin{equations*}
 |h^2 \Delta_\Ss \rho_1(h^2\Delta) u| + h^{1/3} \sum_{j=1}^n |hX_j \rho_1(h^2\Delta) u|_{H^{1/3}_h} + h^{2/3}|\rho_1(h^2\Delta) u|_{H^{2/3}_h} \\
\leq C |P \rho_1(h^2\Delta) u| + O(h^{2/3}) |\rho_1(h^2\Delta) u|_{H^{2/3}_h}.
\end{equations*}
Let $\trho_1 \in C_0^\infty$, be equal to $1$ on $\supp(\rho_1)$ and $0$ where $\rho_2 \neq 1$. Since $\Delta$ and $\Delta_\Ss$ commute, we have $P \rho_1(h^2\Delta) = \rho_1(h^2\Delta) (P-\lambda h) + \lambda h \rho_1(h^2\Delta)+[\frac{h}{i}H_1,\rho_1(h^2\Delta)]$. Both $\lambda h \rho_1(h^2\Delta)$ and $[\frac{h}{i}H_1,\rho_1(h^2\Delta)]$ have wavefront set contained in the elliptic set of $\trho_1(h^2\Delta)$. Therefore, 
\begin{equations*}
C |P \rho_1(h^2\Delta) u| + O(h^{2/3}) |\rho_1(h^2\Delta) u|_{H^{2/3}_h} \\ \leq C |\rho_2(h^2\Delta) (P-\lambda h)  u| + O(h^{2/3}) |\trho_1(h^2\Delta) u|_{H^{2/3}_h} + O(h^\infty)|u|.
\end{equations*}
Hence the theorem follows from a bound on $h^{2/3} |\trho_1(h^2\Delta) u|_{H^{2/3}_h}$. After lifting to $O^*\Mm$ and using \eqref{eq:1b} and \eqref{eq:7f} it suffices to show that
\begin{equation}\label{eq:7i}
v \in C^\infty(O^*\Mm) \ \Rightarrow \ h^{2/3} |\trho_1(h^2\Delta_O) v|_{\tH^{2/3}_h} \leq C|\rho_2(h^2\Delta_O)(\tP-\lambda h)v| + O(h^N)|v|.
\end{equation}

\textbf{Step 2.} Since $\WF_h(\trho_1(h^2 \Delta))$ is a compact subset of $\oT^*O^*\Mm \setminus 0$, there exists a finite collection of points $(x_1,\xi_1), ..., (x_\nu,\xi_\nu) \in \oT^*O^*\Mm$ and open sets $W_{x_1, \xi_1}, ..., W_{x_\nu,\xi_\nu}$ given by Proposition \ref{prop:1} such that 
\begin{equation}\label{eq:9z}
\WF_h(\trho_1(h^2 \Delta)) \subset \bigcup_{k=1}^\nu W_{x_k,\xi_k}.
\end{equation} 
Let $\Psi_{h,k}^m$ be the set of operators in $\Psi^m_h$ with wavefront set contained in $W_{x_k,\xi_k}$. Using \eqref{eq:9z} and a microlocal partition of unity, we can construct operators $E_k \in \Psi_{h,k}^{2/3}$ with
\begin{equation}\label{eq:0q}
v \in C^\infty(O^*\Mm) \ \Rightarrow \ |\trho_1(h^2 \Delta_O) v|_{\tH^{2/3}_h} \leq \sum_{k=1}^\nu |E_k v| + O(h^\infty)|v|.
\end{equation}
Below we obtain bounds on the terms $|E_k v|$.

\textbf{Step 3.} Let $\delta = 1/15$ and $m \leq 2/3$. We first claim that for every $A \in \Psi_{h,k}^m$, there exist $B_1 \in \Psi_{h,k}^{m-2/3}$ and $A' \in \Psi_{h,k}^{m-\delta}$ with
\begin{equation}\label{eq:0e}
h^{2/3}|A v| \leq C|\tP B_1 v|+O(h)|A' v| + O(h^\infty)|v|.
\end{equation}
The operator $\Lambda_{-2/3}A \Lambda_{-m+2/3}$ belongs to $\Psi^{0}_{h,k}$. Proposition \ref{prop:1} gives an operator $B \in \Psi_{h,k}^0$ such that
\begin{equation*}
h^{2/3}|\Lambda_{-2/3}A \Lambda_{-m+2/3} v|_{\tH^{2/3}_h} \leq |\tP Bv| + O(h) |v|_{\tH^{3/5}_h}.
\end{equation*}
Pick $B' \in \Psi_{h,k}^0$ with $\WF_h(B' - \Id) \cap \WF_h(A) = \emptyset$ and replace $v$ by $\Lambda_{m-2/3} B'v$:
\begin{equation*}
h^{2/3}|AB'v| \leq C|\tP B\Lambda_{m-2/3} B'v| + O(h) |\Lambda_{m-2/3} B' v|_{\tH^{3/5}_h}.
\end{equation*}
Since $h^{2/3}|A (\Id - B') v| = O(h^\infty) |v|$, \eqref{eq:0e} holds with $B_1 \de  B \Lambda_{m-2/3} B' \in \Psi^{m-2/3}_{h,k}$ and $A' \de  \Lambda_{3/5} \Lambda_{m-2/3} B' \in \Psi^{m-\delta}_{h,k}$.

\textbf{Step 4.} The goal is now to iterate \eqref{eq:0e}. We first need a commutator-like estimate. For $B_1$ belongs to $\Psi_{h,k}^{m-2/3}$,
\begin{equations*}
\tP  B_1 = B_1 \tP  + [\tP ,B_1] = B_1 (\tP -\lambda h) + 2h \sum_{k,\ell} V_{k\ell}[hV_{k\ell},B_1] + h \Psi_{h,k}^{m-2/3} \\
    = B_1 (\tP -\lambda h) + 2h \sum_{k,\ell} \Lambda_{1/3}hV_{k\ell} \cdot \Psi_{h,k}^{m-1} + h \Psi_{h,k}^{m-2/3}.
\end{equations*}
Hence there exist operators $B_2 \in \Psi_{h,k}^{m-1}$ and $C_0 \in \Psi_{h,k}^{m-2/3}$ such that
\begin{equations*}
|\tP  B_1 v| \leq  C|B_1 (\tP -\lambda h)v| + h \sum_{k,\ell} |hV_{k\ell} B_2 v|_{\tH^{1/3}_h}+ h |C_0 v| \\ \leq C|\rho_2(h^2\Delta_O) (\tP -\lambda h)v| + h^{4/3}|B_2 v|_{\tH^{2/3}_h}+ h^{2/3} |\tP B_2 v| + h |C_0 v| + O(h^\infty)|v|.
\end{equations*}
In the second line we used Lemma \ref{lem:1a} and the elliptic estimate. The slightly weaker bound holds: there exist $B_2 \in \Psi_{h,k}^{m-1}$ and $C_1 \in \Psi_{h,k}^{m-1/3}$ such that
\begin{equation}\label{eq:0s}
|\tP  B_1 v| \leq C|\rho_2(h^2\Delta_O) (\tP -\lambda h)v| + h^{2/3} |\tP  B_2 v| + h |C_1 v| + O(h^\infty)|v|.
\end{equation}
Iterate \eqref{eq:0s} to obtain $B_N \in \Psi^{m-2/3-N/3}$ and $C_N \in \Psi_{h,k}^{m-1/3}$ such that 
\begin{equation*}
|\tP  B_1 v| \leq C|\rho_2(h^2\Delta_O) (\tP -\lambda h) v| + h^{2N/3} |\tP  B_N v| + h |C_N v| + O(h^\infty)|v|.
\end{equation*}
For $N \geq 6$ the operator $\tP B_N$ belongs to $\Psi_h^0$ and $|\tP B_N v| = O(|v|)$. It follows that for $N$ large enough,
\begin{equation}\label{eq:9w}
|\tP  B_1 v| \leq C|\rho_2(h^2\Delta_O) (\tP -\lambda h) v| + h |C_{2N} v| + O(h^N)|v|.
\end{equation}

\textbf{Step 5.} The estimate \eqref{eq:9w} combined with \eqref{eq:0e} show that for every $A_1 \in \Psi^m_{h,k}$ there exists $A_2 \in \Psi^{m-\delta}_{h,k}$ with
\begin{equation*}
h^{2/3}|A_1 v| \leq C |\rho_2(h^2\Delta_O) (\tP -\lambda h) v| + h |A_2 v| + O(h^N)|v|.
\end{equation*}
Here again we can iterate this inequality sufficiently many times to obtain
\begin{equation}\label{eq:0r}
h^{2/3}|A_1 v| \leq C |\rho_2(h^2\Delta_O) (\tP -\lambda h) v| + O(h^N)|v|.
\end{equation}
Recall that $\trho_1$ is controlled by operators microlocalized inside $W_{x_k,\xi_k}$ thanks to \eqref{eq:0g}. Apply \eqref{eq:0r} with $A_1 = E_k$, $k=1, ..., \nu$ and sum over $k$ to get \eqref{eq:7i}:
\begin{equation*}
h^{2/3}|\trho_1(h^2 \Delta_O) v|_{\tH^{2/3}_h} \leq C |\rho_2(h^2\Delta_O) (\tP -\lambda h) v| + O(h^N)|v|.
\end{equation*}
This ends the proof of the theorem.\end{proof}

\subsection{Proof of the subelliptic estimate}\label{subsec:3.3}
In this subsection we show Proposition \ref{prop:1}. We fix $(x_0,\xi_0) \in \oT^*O^*\Mm \setminus 0$.  We distinguish three cases: whether $(x_0,\xi_0) \in \Ell_h(h^2\Delta^V_O)$ -- in this case $\tP$ is elliptic at $(x_0,\xi_0)$ -- or $(x_0,\xi_0) \in \Ell_h(h\tH_1)$ -- in this case $\Im(\tP)$ is elliptic at $(x_0,\xi_0)$ -- or $(x_0,\xi_0) \notin \Ell_h(h\tH_1) \cup \Ell_h(h\tH_1)$. The latter is the hardest; we will use that one of the commutators $[hV_{1\ell},h\tH_1]$ is elliptic at $(x_0,\xi_0)$.

\begin{proof}[Proof of Proposition \ref{prop:1} in the case $(x_0,\xi_0) \in \Ell_h(h^2\Delta_O^V)$] In this case $(x_0,\xi_0) \in \Ell_h(\tP )$. Let $W_{x_0,\xi_0}$ be an open neighborhood of $(x_0,\xi_0)$ contained in $\Ell_h(\tP )$, and $A \in \Psi_h^0$ with wavefront set contained in $W_{x_0,\xi_0}$. Let $B \in \Psi_h^0$ elliptic on $\WF_h(A)$ and with wavefront set contained in $W_{x_0,\xi_0}$. The operator $\tP B$ is elliptic on the wavefront set of $A$. The elliptic estimate \cite[Theorem E.32]{DZ} shows that for $h$ small enough,
\begin{equation*}
v \in C^\infty(O^*\Mm) \ \Rightarrow \ h^{2/3}|Av|_{\tH^{2/3}_h} \leq C |\tP Bv|+O(h^\infty)|v|.
\end{equation*} 
This shows the proposition in this case.
\end{proof}

\begin{proof}[Proof of Proposition \ref{prop:1} in the case $(x_0,\xi_0) \in \Ell_h(h \tH_1)$] Without loss of generalities $(x_0,\xi_0) \in \Ell_h(h \tH_1) \setminus \Ell_h(h^2 \Delta_O^V)$. In particular $V_{1\ell}$ is characteristic at $(x_0,\xi_0)$ for any $\ell$. Let $\sigma_{\tH_m}, \sigma_{V_{k\ell}}$ be the principal symbols of $\frac{h}{i}\tH_m, \frac{h}{i}V_{k\ell}$ given accordingly by \eqref{EQ:0a}. We can find an open neighborhood $W_{x_0,\xi_0} \subset \Ell_h(h\tH_1)$ of $(x_0, \xi_0)$ in $\oT^*O^*\Mm$ such that on $W_{x_0,\xi_0} \cap T^*O^*\Mm$, $\sigma_{\tH_1}^2 - 2\sigma_{V_{1\ell}} \sigma_{\tH_\ell} \geq 0$. Let $A \in \Psi_h^0$ with wavefront set contained in $W_{x_0,\xi_0}$ and $B \in \Psi_h^0$ elliptic on $\WF_h(A)$, with wavefront set contained in $W_{x_0,\xi_0}$, and principal symbol $\sigma_B$. The operator $h \tH_1 B$ is elliptic on $\WF_h(A)$ and \cite[Theorem E.32]{DZ} shows that
\begin{equation*}
|Av|_{\tH^{2/3}_h} \leq C|h \tH_1 B v| + O(h^\infty)|v|.
\end{equation*}
It remains to control $|h \tH_1 B v|$. Using that $\tP$ is equal to $h^2 \Delta_O^V + h \tH_1$ with $\Delta_O^V$ selfadjoint and $\tH_1$ anti-selfadjoint,
\begin{equations}\label{eq:0t}
|\tP Bv|^2 = |h^2 \Delta_O^V B v|^2 + |h \tH_1 B v|^2 + \lr{[h^2 \Delta_O^V,h\tH_1]Bv,Bv} \\
= |h^2 \Delta_O^V B v|^2 + |h \tH_1 B v|^2 + 2h \Re(\lr{B^*\Re(h^2 V_{1\ell} \tH_\ell)Bv,v}),
\end{equations}
where we used that $\Re([h^2 \Delta_O^V,h\tH_1]) = 2h\sum_\ell \Re(h^2V_{1\ell} \tH_\ell)$.
On $W_{x_0,\xi_0} \cap T^*O^*\Mm$, $\sigma_{\tH_1}^2 - 2\sigma_{V_{1\ell}} \sigma_{\tH_\ell} \geq 0$; hence $2|\sigma(B)|^2 i\sigma_{V_{1\ell}} i\sigma_{\tH_\ell} \geq |\sigma_B|^2 (i\sigma_{\tH_1})^2$. The sharp G$\mathring{\text{a}}$rding inequality (see \cite[Proposition E.34]{DZ}) shows that
\begin{equations*}
2h \Re(\lr{B^*\Re(h^2 V_{1\ell} \tH_\ell)Bv,v}) \geq \lr{B^* h^2 \tH_1^2 B v, v} - O(h) |v|^2_{\tH^{1/2}_h}  = -|h \tH_1 B v|^2 - O(h) |v|^2_{\tH^{1/2}_h}.
\end{equations*}
Plug this inequality in \eqref{eq:0t} to obtain
\begin{equation*}
v \in C^\infty(O^*\Mm) \ \Rightarrow \ |Av|^2_{\tH^{2/3}_h} \leq C|h\tH_1 B v|^2 + O(h^\infty)|v|^2 \leq C|\tP Bv|^2 + O(h^2) |v|^2_{\tH^{1/2}_h}.
\end{equation*}
This shows the proposition in the case $(x_0,\xi_0) \in \Ell_h(h\tH_1)$.
\end{proof}

\begin{proof}[Proof of Proposition \ref{prop:1} in the case $(x_0,\xi_0) \notin \Ell_h(\Delta_O^V) \cup \Ell_h(h \tH_1)$.] In this case $(x_0,\xi_0) \notin \Ell_h(hV_{k\ell})$ for any $k,\ell$. Since $\{ V_{k\ell}, \tH_m\}$ span $TO^*\Mm$, there exists $m$ such that $(x_0,\xi_0) \in \Ell_h(h\tH_m)$; and for $(x,\xi) \in T^*O^*\Mm$ in a neighborhood of $(x_0,\xi_0)$, $\sigma_{\tH_m}(x,\xi) \neq 0$. Changing $V_{1m}$ to $-V_{1m}$ does not change $\tP $; and under this change $\tH_m = [V_{1m},\tH_1]$ becomes $-\tH_m$. Hence we can assume without of generalities that $\sigma_{\tH_m}(x,\xi) > 0$ for $(x,\xi) \in T^*O^*\Mm$ in a neighborhood of $(x_0,\xi_0)$. 

We subdivide the proof it in 7 short steps. In the first step we localize the functions and operators involved in a small neighborhood of $x_0$, diffeomorphic to $\R^{d(d+1)/2}$. It allows us to use the class $\Psi_h$ introduced in \S \ref{subsec:2.2} and to perform a second microlocalization in the steps 2 and 3. Step 4 is the main argument. Instead of using an energy estimate obtained after a microlocal reduction as in \cite{L} we apply a positive commutator estimate. This allows us to control microlocally $u$ over certain small frequencies. In step 5 we use the spectral theorem to control microlocally $u$ over the remaining frequencies. In step 6 we combine the results of steps 4,5 to conclude the proof modulo an error term which is shown to be negligible in step 7.

\textbf{Step 1.} The first step in the proof is a localization process. We fix $W_{x_0,\xi_0}$ an open neighborhood of $(x_0,\xi_0)$ in $\oT^*O^*\Mm$. We assume that $W_{x_0,\xi_0}$ is small enough, so that for all $(x,\xi) \in W_{x_0,\xi_0} \cap T^*O^*\Mm$, $\sigma_{\tH_m}(x,\xi) > c|\xi|_g$, $c >0$; and so that there exists a smooth diffeomorphism $\gamma : \UU \subset \R^{d(d+1)/2-1}_y \times \R_\te \rightarrow U \de \{x \in O^*\Mm : \exists \xi, (x,\xi) \in W_{x_0,\xi_0}\}$ such that $d\gamma(\p_\te|_\UU) = V_{1m}|_U$. Let $\Gamma : T^*\UU \rightarrow T^*U$ be the symplectic lift of $\gamma$.

Let $\VV_{k\ell} \de \frac{1}{2}(d\gamma^{-1} V_{k\ell}|_U - (d\gamma^{-1} V_{k\ell}|_U)^*)$. This is an anti-selfadjoint differential operator on $\UU$ which has the same principal symbol as $d\gamma^{-1} V_{k\ell}|_U$. In particular there exists a function $f_{k\ell} \in C^\infty(O^*\Mm)$ such that
\begin{equation}\label{eq:0a}
\VV_{k\ell} = d\gamma^{-1}V_{k\ell}|_U + \gamma^* f_{k\ell}|_U.
\end{equation}
Extend $\VV_{k\ell}$ to an anti-selfadjoint differential operator of order $1$ on $\R^{d(d+1)/2}$ with coefficients in $C^\infty_b(\R^{d(d+1)/2})$ -- with $\VV_{1m}$ specifically continued by $\p_\te$ -- and define $\LL_O^V \de -\sum_{k,\ell} \VV_{k\ell}^2$. Since $\Delta_O^V$ commutes with $V_{1m}$, $[\LL_O^V,D_\te] w = 0$ for each $w \in C^\infty(\R^{d(d+1)/2})$ supported on $\UU$. 

Similarly, we define $\HH_1 \de \frac{1}{2}(d\gamma^{-1} \tH_1 |_U - (d\gamma^{-1} \tH_1 |_U)^*)$, which is an anti-selfadjoint differential operator on $\UU$. It satisfies
\begin{equation}\label{eq:0b}
\HH_1|_\UU = d\gamma^{-1} \tH_1 |_U + \gamma^*f|_U,
\end{equation}
for a certain function $f \in C^\infty(O^*\Mm)$. It extends to an anti-selfadjoint differential operator of order $1$ on $\R^{d(d+1)/2}$ with coefficients in $C^\infty_b(\R^{d(d+1)/2})$. We define $\PP \de h^2\LL_O^V + h \HH_1$. 

Let $A \in \Psi_h^0$ with $\WF_h(A) \subset W_{x_0,\xi_0}$ and $\psi \in C^\infty(O^*\Mm)$ be equal to $1$ on the set $\{x \in O^*\Mm, \exists \xi, (x,\xi) \in \WF_h(A) \}$ and $0$ outside $U$. The function $1-\psi$ can be seen as a pseudodifferential operator in $\Psi_h^0$ with $\WF_h(1-\psi) \cap W_{x_0,\xi_0} = \emptyset$. In particular $A(1-\psi) \in h^\infty \Psi_h^{-\infty}, (1-\psi) A \in h^\infty \Psi_h^{-\infty}$ and to prove the proposition it suffices to show that
\begin{equation}\label{eq:7l}
v \in C^\infty(O^*\Mm) \ \Rightarrow \ h^{2/3}|\psi A \psi^2 v|_{\tH^{2/3}_h} \leq C|\tP  \psi A \psi^2 v|+O(h) |v|_{\tH^{3/5}_h}. 
\end{equation}
We define $(\gamma^*)^{-1}$ (resp. $\gamma^*$) the operator defined on functions on $\UU$ (resp. $O^*\Mm$) by
\begin{equation*}
(\gamma^*)^{-1}w (x) = \systeme{ w(\gamma^{-1}(x)) \text{ if } x \in U \\ 0 \ \ \ \ \text{ otherwise}} \ \ \ \left(\text{resp. } \gamma^*v (z) = \systeme{ v(\gamma(z)) \text{ if } z \in \UU \\ 0 \ \ \ \ \text{ otherwise}}\right).
\end{equation*}
The function $\psi A \psi^2 u$ has support in $U$; the operator $\AAA \de  \gamma^* \psi A \psi (\gamma^*)^{-1}$ is a pseudodifferential operator in $\Psi^0_h$ on $\R^3$ with wavefront set in $\Gamma^{-1} (W_{x_0,\xi_0})$; and 
\begin{equation}\label{eq:7k}
|\psi A \psi^2 v|_{\tH^{2/3}_h} \leq C |\gamma^* \psi A \psi^2 v|_{\tH^{2/3}_h} = C |\AAA v|_{\tH^{2/3}_h}, \ \ w \de \gamma^* \psi v. 
\end{equation}
Thanks to \eqref{eq:0a}, \eqref{eq:0b},
\begin{equations*}
\PP \gamma^* \psi = -\gamma^* \sum_{k,\ell}  h^2(V_{k\ell} + f_{k\ell})^2  \psi + \gamma^*h(\tH_1 + f) \psi  = \gamma^* \tP \psi - 2h \gamma^* \sum_{k,\ell} f_{k\ell} h V_{k\ell} \psi + h\gamma^* g \psi,
\end{equations*}
where $g \de f- h \sum_{k,\ell} f_{k\ell}^2 + (V_{k\ell} f_{k\ell})$ belongs to $C^\infty(O^*\Mm)$. It follows that $|\PP \AAA v|$
\begin{equations}\label{eq:7j}
\leq |\tP \psi A \psi^2 v| + O(h)\sum_{k,\ell} |h V_{k\ell} \psi A \psi^2 v| + O(h) |v| 
\leq 2 |\tP \psi A \psi^2 u| + O(h) |v|.
\end{equations}
In the last inequality we used that $\Re(\tP) = h^2\Delta_O^V = -\sum_{k,\ell} (hV_{k\ell})^2$ hence $|hV_{k\ell} v|^2 \leq |\tP v| |v|$. Finally we observe that since $w = \gamma^*\psi v$, $|w|_{\tH^{3/5}_h} = |\gamma^* \psi v|_{\tH^{3/5}_h} \leq C |v|_{\tH^{3/5}_h}$. Thanks to \eqref{eq:7k} and \eqref{eq:7j} the bound \eqref{eq:7l} will follow from the estimate
\begin{equation}\label{eq:7}
w \in C^\infty_0(\R^{d(d+1)/2}), \ \supp(w) \subset \UU \ \Rightarrow \ h^{2/3} |\AAA w|_{\tH^{2/3}_h} \leq C |\PP \AAA w| + O(h) |w|_{\tH_h^{3/5}}. 
\end{equation}
We have reduced the estimate on $O^*\Mm$ to an estimate on $\R^{d(d+1)/2}$. In the following steps we prove \eqref{eq:7}.

\textbf{Step 2.} Let $\chi, \chi_0 \in C_0^\infty(\R^{d(d+1)/2})$ be two functions such that $\chi$ is supported away from $0$, $\WF_h(\AAA) \cap \WF_h(\chi_0(hD)) = \emptyset$, and
\begin{equation*}
1 = \sum_{j=0}^\infty \chi_j(\xi), \ \ \chi_j(\xi) \de  \chi(2^{-j}\xi) \text{ for } j \geq 1.
\end{equation*}
Write a Littlewood-Paley decomposition of $\AAA$:
\begin{equation*}
\AAA = \sum_{j=0}^\infty \AAA_j, \ \ \AAA_j \de  \chi_j(hD) \AAA.
\end{equation*}  
Given $a$ a symbol on $\R^{d(d+1)/2} \times \R^{d(d+1)/2}$ we denote by $\Op_h(a)$ the standard quantization of $a$ -- see \cite[\S 4]{Z}. The following lemma studies the composition of a pseudodifferential operator with symbol in $S^m$ with a dyadic decomposition:

\begin{lem}\label{lem:1} If $a \in S^m$, both the operators $2^{-jm} \Op_h(a) \chi(2^{-j} h D)$ and $2^{-jm} \chi(2^{-j} h D) \Op_h (a)$ belong to $\Psi_{2^{-j} h}$, with semiclassical symbol $a_j \chi + 2^{-j}h \cdot S$. \end{lem}

\begin{proof} We first note that if $a_j(x,\xi) \de 2^{-jm} a(x,2^j \xi)$ then 
\begin{equation*}
2^{-jm} \Op_h(a) \chi(2^{-j}h D) = \Op_{2^{-j}h} (a_j \# \chi)= \Op_{2^{-j}h} (a_j \chi).
\end{equation*}
It suffices to show that the $S$-seminorms of $a_j \chi$ are uniformly bounded in $j$. We have 
\begin{equation}\label{Eq:6}
\left|\p_x^\az \p_{\xi}^\beta a_j(x,\xi) \chi(\xi)\right| \leq C_{\az \beta} \sup_{2^j \xi \in \supp(\chi)} 2^{-jm+j|\beta|} \lr{2^j \xi}^{m-|\beta|}.
\end{equation}
Since $\supp(\chi)$ is a compact subset of $\R^3 \setminus 0$, the  right hand side of \eqref{Eq:6} is uniformly bounded in $j$. This shows that $a_j \# \chi = a_j \chi \in S$, hence $2^{-jm} \Op_h(a) \chi(2^{-j}h D)$ belongs to $\Psi_{2^{-j}h}$ with symbol $a_j \chi$. The operator $2^{-jm} \chi(2^{-j}h D)\Op_h(a)$ is the adjoint of $2^{-jm} \Op_h(a^*)\chi(2^{-j}h D)$, thus it also belongs to $\Psi_{2^{-j}h}$. By the composition formula for symbols of semiclassical operators, its semiclassical symbol is equal to  $a_j \chi + 2^{-j}h \cdot S$.
\end{proof}

A direct application of this result shows that $\AAA_j$ belongs to $\Psi_{2^{-j}h}$. In addition, $\AAA_0 \in h^\infty \Psi^{-\infty}_h$, which implies immediately $|\AAA_0 w| \leq O(h) |w|_{\tH^{3/5}_h}$. We obtain in the next steps estimates on $|\AAA_j w|$ for $j \geq 1$.

\textbf{Step 3.} We start with a simple result:

\begin{lem}\label{lem:2} There exist functions $\Phi \in C^\infty_b(\R)$ and $\phi \in C_0^\infty(\R)$ such that $\phi(0) = 1$, $\phi \geq 0$ and $\phi^2 = (\Phi^2)'$.
\end{lem}

\begin{proof} It is enough to construct $\phi$ with $\phi(0) > 0$ then to multiply $\Phi, \phi$ by a suitable multiplicative constant. Let $\Phi$ be a smooth non-decreasing function with
\begin{equation*}
\Phi(x) = \systeme{0 \ \ \ \ \ \ \ \text{ if } x \leq -1, \\
 e^{-(x+1)^{-1}} \ \text{ if } x \in [-1,0], \\
 1 \ \ \ \ \ \text{ if } x \geq 1.}
\end{equation*}
If $\phi$ is the non-negative root of $(\Phi^2)'$ then $\phi$ has compact support and $\phi(0) > 0$. Since the $s \in [0,\infty) \mapsto \sqrt{s}$ is smooth everywhere but at $0$, $\phi$ is smooth everywhere but possibly at $-1$. But 
\begin{equation*}
\phi(x) = \systeme{0 \ \ \ \ \ \ \ \ \ \ \ \ \ \ \ \ \ \ \ \ \ \ \ \text{ if } x \leq -1, \\
2^{1/2} (x+1)^{-1}e^{-(x+1)^{-1}} \text{ if } x \in [-1,0],}
\end{equation*}
which is smooth at $x=-1$.\end{proof}

Let $\Phi, \phi$ be given by Lemma \ref{lem:2}. Let $h_j \de h^{2/3} 2^{-j/3}$ and consider the operator $\Phi(h_j D_\te)$. This operator belongs to $\Psi_{h_j}$ with semiclassical symbol $\Phi(\xi_\te)$. Below we show an estimate on $|\AAA_j w|$, by splitting it into two parts, $|\phi(h_j D_\te) w|$ and $|(\Id-\phi(h_jD_\te))w|$.

\textbf{Step 4.} In order to estimate $|\phi(h_j D_\te) w|$ we use a positive commutator argument and the sharp G$\mathring{\text{a}}$rding inequality. Observing that $\sigma_{\tH_m}(x,\xi) > c|\xi|_g$ on $W_{x_0,\xi_0} \cap T^*O^*\Mm$, the principal symbol $\sigma_{\HH_1}$ of $\frac{1}{i} \HH_1$ satisfies
\begin{equations*}
\{ \xi_\te, \sigma_{\HH_1} \}(x,\xi) > c |\xi|_g \ \ \text{ on } \Gamma^{-1}(W_{x_0,\xi_0}) \cap T^*O^*\Mm.
\end{equations*}
Recall that $\PP = h^2 \LL_O^V + h \HH_1$. Similarly to \cite[Equation (2.47)]{L},
\begin{equations*}
\Re(\lr{\PP \AAA_jw,\Phi(h_jD_\te)^2 \AAA_jw}) \\
 = \Re(\lr{h^2 \Phi(h_jD_\te) \LL_O^V \AAA_j w, \Phi(h_jD_\te) \AAA_j w}) + \Re(\lr{h\HH_1 \AAA_jw, \Phi(h_jD_\te)^2 \AAA_jw}).
\end{equations*}
We study the first term. We observe that $\LL_O^V \AAA_j = \LL_O^V \tchi_j(hD) \cdot \AAA_j$. Lemma \ref{lem:1} shows that both the operators $2^{-2j} h^2 \LL_O^V \tchi_j(hD)$ and $\AAA_j$ belong to $\Psi_{2^{-j} h}$. Since $2^{-j} h \leq h_j = h^{2/3} 2^{-j/3}$, they \textit{a fortiori} belong to $\Psi_{h_j}$. In addition, $D_\te$ and $\LL_O^V$ commute on $\UU$ and $\AAA$ has wavefront set contained in $T\UU$. The asymptotic expansion formula for composition of pseudodifferential operators \cite[Theorem 4.14]{Z} show that 
\begin{equation*}
h^2 \Phi(h_jD_\te) \LL_O^V \AAA_j = h^2 \LL_O^V \Phi(h_jD_\te) \AAA_j + h_j^\infty \Psi_{h_j}.
\end{equation*}
Using that $\LL_V^O = -\sum_{k,\ell} \VV_{k\ell}^2 \geq 0$ we get $\Re(\lr{h^2 \Phi(h_jD_\te) \LL_O^V \AAA_j w, \Phi(h_jD_\te) \AAA_j w}) =$
\begin{equation*}
\lr{h^2 \LL_O^V  \Phi(h_jD_\te) \AAA_j w, \Phi(h_jD_\te) \AAA_j w} + O(h_j^\infty)|w| \geq O(h_j^\infty)|w|.
\end{equation*}
We next focus on the term $\Re(\lr{h\HH_1 \AAA_jw, \Phi(h_jD_\te)^2 \AAA_jw})$. Since $h\HH_1$ is anti-selfadjoint, it is equal to $\Re(\lr{[h\HH_1,\Phi(h_jD_\te)]\AAA_jw,\Phi(h_jD_\te) \AAA_jw})$. The real part of the operator $\Phi(h_jD_\te)[h\HH_1,\Phi(h_jD_\te)]$ is equal to $\frac{1}{2}[h\HH_1,\Phi(h_jD_\te)^2]$. We obtain
\begin{equation}\label{eq:2}
\Re(\lr{\PP \AAA_jw,\Phi(h_jD_\te)^2 \AAA_jw}) \geq \frac{1}{2}\lr{\AAA_j^*[h\HH_1,\Phi(h_jD_\te)^2]\AAA_j w, w} + O(h_j^\infty)|w|.
\end{equation} 

We now study the commutator term $\EE_j \de 2^{-j} \AAA_j^*[h\HH_1,\Phi(h_jD_\te)^2]\AAA_j$. We claim that it belongs to $\Psi_{h_j}$. To show this claim we fix $\tchi \in C_0^\infty(\R^{d(d+1)/2} \setminus 0)$ equal to $1$ near $\supp(\chi)$ and we write
\begin{equations*}
\EE_j = 2 \Re \left( (\AAA \tchi_j(hD)) \cdot (2^{-j}\chi_j(hD) h \HH_1) \cdot \Phi(h_j D_\te)^2 \cdot \AAA_j \right).
\end{equations*}
By Lemma \ref{lem:1}, the operators $\AAA \tchi_j(hD)$, $2^{-j}\chi_j(hD) h \HH_1$ and $\AAA_j$ belong to $\Psi_{2^{-j}h}$. Since $2^{-j} h \leq h_j = h^{2/3} 2^{-j/3}$, they also belong to $\Psi_{h_j}$. The operator $\Phi(h_j D_\te)$ has symbol equal to $\Phi(\xi_\te)$ in the $h_j$-quantization and the composition theorem for semiclassical operators shows that $\EE_j \in \Psi_{h_j}$. 

The semiclassical symbols of $\AAA_j$ and $2^{-j} h \HH_1$ are given modulo $O(h_j) S$ by
\begin{equation*}
a(x,h^{1/3} 2^{j/3} \xi) \chi(h^{1/3} 2^{-2j/3} \xi), \ \ \ \ 2^{-2j/3}h^{1/3} \sigma_{\HH_1},
\end{equation*}
where $a$ is the semiclassical symbol of $\AAA$ in the $h$-quantization. By the composition formula for symbols of semiclassical operators \cite[Theorem 4.14]{Z}, the semiclassical symbol $\sigma_{\EE_j}$ of $\EE_j$ in the $h_j$-quantization is  given modulo $O(h_j^2)S$ by
\begin{equations}\label{eq:1}
\chi(h^{1/3} 2^{-2j/3} \xi)^2 |a(x,h^{1/3} 2^{j/3} \xi)|^2 \cdot \dfrac{h_j}{i} \{ \Phi(\xi_\te)^2, 2^{-2j/3}h^{1/3} \sigma_{\HH_1}\} \\
 = \chi(h^{1/3} 2^{-2j/3} \xi)^2 |a(x,h^{1/3} 2^{j/3} \xi)|^2 \phi(\xi_\te)^2 \cdot 2^{-j}h \{ \xi_\te, \sigma_{\HH_1} \}. % \\  \geq h_j \chi(h^{1/3} 2^{-2j/3} \xi)^2 |a(x,h^{1/3} 2^{j/3} \xi)|^2 \phi(\xi_\te)^2.
\end{equations}
The wavefront set of $\AAA$ (hence the support of $a$) is contained in $\Gamma^{-1}(W_{x_0,\xi_0})$ itself contained in the conical set $\overline{\{ \{ \xi_\te, \sigma_{\HH_1} \} \geq c|\xi| \}}$, and $|\xi| \geq c h^{-1/3} 2^{2j/3}$ whenever $\chi(h^{1/3} 2^{-2j/3} \xi) \neq 0$. It follows that 
\begin{equation*}
\sigma_{\EE_j} \geq \chi(h^{1/3} 2^{-2j/3} \xi)^2 |a(x,h^{1/3} 2^{j/3} \xi)|^2 \phi(\xi_\te)^2 \cdot c h^{2/3} 2^{-j/3}.
\end{equation*}
The sharp G$\mathring{\text{a}}$rding inequality \cite[Theorem 4.32]{Z} implies
\begin{equation*}
\lr{\EE_j w, w} \geq c h^{2/3} 2^{-j/3} |\phi(h_jD_\te) \AAA_j w|^2 + O(h_j^2) |w|^2.
\end{equation*}
Since $h_j = h^{2/3} 2^{-j/3}$ and $\EE_j = 2^{-j}\AAA_j^*[h\HH_1,\Phi(h_jD_\te)^2]\AAA_j$ this yields
\begin{equation*}
\lr{\AAA_j^*[h\HH_1,\Phi(h_jD_\te)^2]\AAA_j w, w} \geq c h^{2/3} 2^{2j/3} |\phi(h_jD_\te) \AAA_j w|^2 + O(h^{4/3} 2^{j/3}) |w|^2.
\end{equation*}
Therefore we can come back to \eqref{eq:2} and obtain
\begin{equation*}
\Re(\lr{\PP \AAA_jw,\Phi(h_jD_\te)^2 \AAA_jw}) \geq c h^{2/3} 2^{2j/3} |\phi(h_jD_\te) \AAA_j w|^2 + O(h^{4/3} 2^{j/3}) |w|^2.
\end{equation*}
Since $\Phi$ is uniformly bounded, the operator $\Phi(h_jD_\te)^2$ is bounded on $L^2$. This gives the estimate on $|\phi(h_jD_\te) \AAA_j w|$:
\begin{equation*}
h^{2/3} 2^{2j/3} |\phi(h_jD_\te) \AAA_j w|^2  \leq C|\PP \AAA_jw||\AAA_jw| +  O(h^{4/3} 2^{j/3}) |w|^2.
\end{equation*}

\textbf{Step 5.} The estimate on $|(\Id-\phi(h_j D_\te)) w|$ follows from the spectral theorem. Since $\phi(0) = 1$ there exists a smooth bounded function $\varphi$ such that $1-\phi(t) = t\varphi(t)$. The operator $\varphi(h_jD_\te)$ is uniformly bounded on $L^2$ hence 
\begin{equation}\label{eq:7n}
h^{2/3} 2^{2j/3} |(\Id-\phi(h_j D_\te)) \AAA_j w|^2 = h^{2/3} 2^{2j/3} |\varphi(h_jD_\te) h_jD_\te \AAA_jw|^2 \leq C |h D_\te \AAA_jw|^2.
\end{equation}
We recall that $\p_\te = \VV_{1m}$ and that $\LL_O^V = -\sum_{k,\ell} \VV_{k\ell}^2 \geq -\VV_{1m}^2$; hence
\begin{equations*}
h^{2/3} 2^{2j/3} |(\Id-\phi(h_j D_\te)) \AAA_j w|^2 \leq C \lr{\LL_O^V \AAA_jw, \AAA_j w} \leq C|\PP \AAA_jw| |\AAA_jw|.
\end{equations*}

\textbf{Step 6.} Combining the results of the steps 4 and 5, we obtain the estimate 
\begin{equation*}
h^{2/3} 2^{2j/3}|\AAA_j w|^2 \leq C|\PP \AAA_j w| |\AAA_j w|+O(h^{4/3}2^{j/3}) |w|^2.
\end{equation*}
Let $\tchi \in C_0^\infty(\R^{d(d+1)/2} \setminus 0)$ equal to $1$ on $\supp(\chi)$. We apply the above estimate to $\tchi_j(hD) w$ and we observe that both $\AAA_j$ and $\Id - \tchi_j(hD)$ belong to $\Psi_{2^{-j}h}$ and that their symbols have disjoint supports; therefore $|\AAA_j (\Id - \tchi_j(hD))w| = O(h^\infty 2^{-j\infty})|w|$ by the composition theorem. Similarly by Lemma \ref{lem:1}, $2^{-2j}\PP \AAA_j$ belongs to $\Psi_{2^{-j}h}$ and its symbol has disjoint support from the one of $\Id - \tchi_j(hD)$; therefore $2^{-2j}|\PP \AAA_j (\Id - \tchi_j(hD))w| = O(h^\infty2^{-j\infty})|w|$. It follows that 
\begin{equation*}
h^{2/3} 2^{2j/3}|\AAA_j w|^2 \leq C|\PP \AAA_j w| |\AAA_jw|+O(h^{4/3}2^{j/3}) |\tchi_j(hD) w|^2 + O(h^\infty 2^{-j\infty})|w|^2.
\end{equation*}
The inequality $ab \leq a^2 + b^2$ and the identity $\AAA_j = \chi_j(hD) \AAA$ shows that
\begin{equations}\label{eq:3}
h^{4/3} 2^{4j/3}|\chi_j(hD) \AAA w|^2 \leq C|\PP \AAA_j w|^2+ O(h^2 2^j) |\tchi_j(hD) w|^2 + O(h^\infty 2^{-j\infty})|w|^2\\
\leq C|\chi_j(hD) \PP \AAA w|^2+ C |[\PP ,\chi_j(hD)]\AAA w|^2 + O(h^2 2^j) |\tchi_j(hD) w|^2 + O(2^{-j}h^2)|w|^2.
\end{equations}

\textbf{Step 7.} To conclude we show the commutator term $|[\PP ,\chi_j(hD)]\AAA w|$ in the right hand side of \eqref{eq:3} is negligible. Recall that $\PP = -h^2 \sum_{k,\ell} \VV_{k\ell}^2 + h \HH_1$ and write
\begin{equation*}
[\PP ,\chi_j(hD)] =  [h\HH_1,\chi_j(hD)]-\sum_{k,\ell} 2h \VV_{k\ell} [h\VV_{k\ell}, \chi_j(hD)] + [h\VV_{k\ell}, [h\VV_{k\ell}, \chi_j(hD)]].
\end{equation*}

We first control the term $|[h\HH_1,\chi_j(hD)]\AAA w|$. We can write
\begin{equations*}
2^{-j/2}[h\HH_1,\chi_j(hD)]\lr{hD}^{-1/2} \\ = 2^{-j} h\HH_1 \tchi_j(hD)\cdot 2^{j/2} \chi_j(hD) \lr{hD}^{-1/2} -  2^{j/2}  \chi_j(hD)\lr{hD}^{-1/2} \cdot 2^{-j} h\HH_1 \tchi_j(hD).
\end{equations*}
By Lemma \ref{lem:1}, both $2^{-j} h\HH_1 \tchi_j(hD)$ and $2^{j/2}  \chi_j(hD)\lr{hD}^{-1/2}$ belong to $\Psi_{2^{-j}h}$. It follows that the operator $2^{-j/2} [h\HH_1, \chi_j(hD)] \lr{hD}^{-1/2}$ belongs to $\Psi_{2^{-j} h}$. Its symbol in the $2^{-j}h$-quantization is given by the asymptotic formula and has vanishing leading term; therefore $2^{-j/2} [h\HH_1, \chi_j(hD)] \lr{hD}^{-1/2}$ belongs to $2^{-j} h \Psi_{2^{-j}h}$. As such it is bounded on $L^2$ with norm $O(2^{-j} h)$. This yields
\begin{equation}\label{eq:0d}
|[h\HH_1, \chi_j(hD)] \AAA w| = O(2^{-j/2} h) |w|_{\tH^{1/2}_h} = O(2^{-j/2} h) |w|_{\tH^{3/5}_h}.
\end{equation}
%\tr{maybe incorporate the argument to Lemma \ref{lem:1} so that the proof is more direct.}

By arguments similar to the one needed to show \eqref{eq:0d}, $[h\VV_{k\ell}, [h\VV_{k\ell}, \chi_j(hD)]] \lr{hD}^{-1/2}$ belongs to $2^{-j/2} h^2 \Psi_{2^{-j} h}$ and 
\begin{equation}\label{eq:0y}
|[h\VV_{k\ell}, [h\VV_{k\ell}, \chi_j(hD)]] \AAA w| = O(2^{-j/2} h^2) |w|_{\tH^{1/2}_h}= O(2^{-j/2} h^2) |w|_{\tH^{3/5}_h}.
\end{equation}

The term $h \VV_{k\ell} [h\VV_{k\ell}, \chi_j(hD)]$ requires some extra work. Fix $j,k,\ell$ and define $B = [h\VV_{k\ell}, \chi_j(hD)]$. Then,
\begin{equation*}
|h \VV_{k\ell} B \AAA w|^2 \leq \Re(\lr{\PP B \AAA w, B \AAA w}) = \Re(\lr{\PP \AAA w, B^*B \AAA w}) + \Re(\lr{[P,B]\AAA w,B\AAA w}).
\end{equation*}
By the same arguments as needed to show \eqref{eq:0d}, the operator $B^*B \lr{hD}^{-1/2}$ belongs to $2^{-j/2} h^2 \Psi_{2^{-j}h}$. Hence $\Re(\lr{\PP \AAA w, B^*B \AAA w}) = O(2^{-j/2}h^2) |\PP \AAA w| |w|_{\tH^{1/2}_h}$. On the other hand the operator $\lr{hD}^{-3/5}[P,B] \lr{hD}^{-3/5}$ belongs to $2^{-j/5} h^2 \Psi_{h_j}$ and this implies that $\Re(\lr{[P,B]\AAA w,B \AAA w}) = O(2^{-j/5} h^3) |w|_{\tH^{3/5}_h}^2$. Combining all these estimates together we obtain that
\begin{equation}\label{eq:0x}
|h \VV_{k\ell} B \AAA w|^2 = O(2^{-j/5}h^3) |w|^2_{\tH^{3/5}_h} + O(2^{-j/2}h^2) |\PP \AAA w| |w|_{\tH^{3/5}_h}.
\end{equation}

We plug \eqref{eq:0d}, \eqref{eq:0y}, \eqref{eq:0x} in \eqref{eq:3} to obtain the estimate $h^{4/3} 2^{4j/3}|\chi_j(hD) \AAA w|^2 $
\begin{equation*}
\leq C|\chi_j(hD) \PP \AAA w|^2+ O(h^2 2^j) |\tchi_j(hD) w|^2 + O(2^{-j/5} h^3) |w|_{\tH^{3/5}_h}^2 + O(2^{-j/2}h^2) |\PP \AAA w| |w|_{\tH^{3/5}_h}.
\end{equation*}
Summation over $j$ allows us to conclude thanks to \cite[Equation (9.3.29)]{Z}:
\begin{equation*}
h^{4/3} |\AAA w|_{\tH^{2/3}_h}^2 \leq C |\PP \AAA w|^2 + O(h^2)|w|_{\tH^{1/2}_h}^2 + O(h^3)|w|_{\tH^{3/5}_h} + O(h^2) |\PP \AAA w||w|_{\tH^{3/5}_h}.
\end{equation*}
This implies \eqref{eq:7}, hence the proof is over.\end{proof}

\section{Subelliptic estimates in Anisotropic Sobolev spaces}

\subsection{Anisotropic Sobolev spaces} To define Pollicott--Ruelle resonances as eigenvalues we need to change the spaces on which $ H_1 $ acts. These spaces originally appeared as anisotropic Sobolev spaces in Baladi \cite{Ba}, Liverani \cite{Liv}, Gou\"ezel--Liverani \cite{GoLi}, Baladi--Tsuji \cite{BaTs}. We follow a microlocal approach due to Faure--Sj\"ostrand \cite{FS} in a version given by Dyatlov--Zworski \cite{DZ1}. It allows the use of PDE methods in the study of the Pollicott--Ruelle spectrum.

For $s,r \in \R$, let $G_{r,s}(h) \in \Psi_h^{0+}$ with principal symbol $\sigma_{G_{r,s}}$ given by
\begin{equation}\label{Eq:3c}
\sigma_{G_{r,s}}(x,\xi) \de  (sm(x,\xi) + r)\rho_0(|\xi|_g) \log(|\xi|_g),
\end{equation}
where $\rho_0 \in C^\infty(\R,[0,1])$ vanishes on $[-1,1]$ and is equal to $1$ on $\R \setminus [-2,2]$ and $m \in C^\infty(T^*S^*\Mm \setminus 0,[-1,1])$ is homogeneous of degree $0$ with
\begin{equation*}
\systeme{ m(x,\xi) = 1 \ \ \text{ near } E_s^* \\ m(x,\xi) = -1 \text{ near } E_u^*
 } \ \ \text{and } \{m,\sigma_{H_1}\} \geq 0. 
\end{equation*}
The existence of $m$ is proved in \cite[Lemma 3.1]{DZ1}. For every $s, r \geq 0$, the operator $e^{G_{r,s}(h)}$ belongs to $\Psi_h^{s+r+}$ and the semiclassical spaces of \cite{DZ1} are defined as $H_h^{r,s} \de e^{-G_{r,s}(h)} L^2$. In particular functions in $H_h^{r,s}$ are in $H^{r+s}_h$ microlocally near $E_s^*$  and in $H^{r-s}_h$ microlocally near $E_u^*$:
\begin{equations}\label{Eq:6d}
A \in \Psi_h^0, \ \WF_h(A) \text{ sufficiently close to } E_s^* \ \Rightarrow \ 
|Au|_{H^{r+s}_h} \leq C|u|_{H_h^{r,s}}, \\ 
A \in \Psi_h^0, \ \WF_h(A) \text{ sufficiently close to } E_u^* \ \Rightarrow \ 
|Au|_{H^{r-s}_h} \leq C|u|_{H_h^{r,s}}.
\end{equations}
In addition if $r,s \in \R$ are fixed and $h > 0$ varies the spaces $H_h^{r,s}$ are equal and there exists a constant $C$ such that
\begin{equation}\label{eq:0j}
C^{-1} h^{|s|+|r|} |u|_{H_1^{r,s}} \leq |u|_{H_h^{r,s}} \leq C h^{-|s|-|r|}|u|_{H_1^{r,s}}.
\end{equation}

\subsection{High frequency estimate in $H_1^{r,s}$} The first result of this section extends the $L^2$-based hypoelliptic estimate of Theorem \ref{thm:1} to anisotropic Sobolev spaces:  

\begin{proposition}\label{prop:1d} For every $R,N \geq 0$ and $r,s \in \R$, $\rho_1, \rho_2$ satisfying \eqref{eq:0g}, there exist $C_{R,N,r,s} > 0$ and  $\epsi_0 >0$ such that for every $0 < \epsi \leq \epsi_0$ and $|\lambda| \leq R$,
\begin{equation}\label{eq:1w}
|\rho_1(\epsi^2\Delta) \epsi^2 \Delta_\Ss u|_{H_1^{r,s}} \leq C_{R,N,r,s}  |\rho_2(\epsi^2\Delta) \epsi(P_\epsi-\lambda) u|_{H_1^{r,s}}+ O(\epsi^N) |u|_{H_1^{r,s}}.
\end{equation}
\end{proposition}

\begin{proof} First observe that as in \cite[Equation (4.4)]{DZ2}, if $B$ is a semiclassical pseudodifferential operator then
\begin{equation*}
\WF_\epsi(B) \subset \oT^*S^*\Mm \setminus 0 \ \Rightarrow \ (e^{G_{r,s}(1)} - e^{G_{r,s}(\epsi)}) B \in \epsi^\infty \Psi^{-\infty}_\epsi.
\end{equation*}
Since $\rho_1(\epsi^2 \Delta), \rho_2(\epsi^2 \Delta)$ are microlocalized away from the zero section  and because of \eqref{eq:0j}, the proposition will follow from the bound
\begin{equation}\label{eq:3a}
|\rho_1(\epsi^2\Delta) \epsi^2 \Delta_\Ss u|_{H_\epsi^{r,s}} \leq C |\rho_2(\epsi^2\Delta) \epsi (P_\epsi-\lambda) u|_{H_\epsi^{r,s}}+ O(\epsi^N) |u|_{H_\epsi^{r,s}}.
\end{equation}
Below we conjugate the operators involved in \eqref{eq:3a} with $e^{G_{r,s}(\epsi)}$ and show a $L^2$-based estimate equivalent to \eqref{eq:3a}. 

For $A \in \Psi_\epsi^m$, let $[A]_{r,s}$ be the operator $e^{G_{r,s}(\epsi)} A e^{-G_{r,s}(\epsi)}$. We have
\begin{equation}\label{eq:3b}
[A]_{r,s} = A + [G_{r,s}(\epsi),A] + \epsi^2 \Psi_\epsi^{m-2+},
\end{equation}
see the equation \cite[(3.11)]{DDZ} and the discussion following it. For $\rho_1, \rho_2$ satisfying \eqref{eq:0g}, let $\trho_1, \trho_2$ be smooth functions satisfying \eqref{eq:0g}, with $\trho_1 = 1$ on $\supp(\rho_1)$ and $\trho_2 = 0$ on $\{\rho_2 \neq 1\}$. We use the identity \eqref{eq:3b} to prove that:
\begin{equation}\label{eq:7da}
\left|\left(\rho_1(\epsi^2 \Delta) \epsi^2 \Delta_\Ss -[\rho_1(\epsi^2 \Delta) \epsi^2 \Delta_\Ss]_{r,s}  \right)v\right| \leq C |\trho_2(\epsi^2 \Delta) \epsi(P_\epsi-\lambda) v| + O(\epsi^N) |v|.
\end{equation}
Since $\Delta_\Ss = -\sum_{j=1}^n X_j^2$, we have
\begin{equations}\label{eq:3c}
\rho_1(\epsi^2\Delta) \epsi^2 \Delta_\Ss - [\rho_1(\epsi^2\Delta) \epsi^2 \Delta_\Ss]_{r,s} \in \sum_{j=1}^n \epsi \Psi_\epsi^{0+} \cdot \epsi X_j + \epsi^2 \Psi_\epsi^{0+},
\end{equations}
where the terms in $\Psi_\epsi^{0+}$ have wavefront sets contained in $\WF_\epsi(\rho_1(\epsi^2 \Delta))$, itself contained in $\Ell_\epsi(\trho_1(\epsi^2 \Delta))$. Thus,
\begin{equations*}
\left|\left(\rho_1(\epsi^2 \Delta) \epsi^2 \Delta_\Ss-[(\rho_1(\epsi^2 \Delta) \epsi^2 \Delta_\Ss]_{r,s}  \right)v\right| \\
\leq O(\epsi) \sum_{j=1}^n |\trho_1(\epsi^2 \Delta) \epsi X_j v|_{H^{1/3}_\epsi} + O(\epsi^2) |\trho_1(\epsi^2 \Delta) v|_{H^{2/3}_\epsi} + O(\epsi^\infty)|v|.
\end{equations*}
Theorem \ref{thm:1} applied with the pair $(\trho_1, \trho_2)$ estimates the right hand side by $C |\trho_2(\epsi^2 \Delta) \epsi (P_\epsi-\lambda) v|+O(\epsi^N)|v|$. This gives \eqref{eq:7da}.

Thanks to \eqref{eq:7da},
\begin{equations}\label{eq:3e}
|[\rho_1(\epsi^2\Delta) \epsi^2 \Delta_\Ss]_{r,s} v| \leq |\rho_1(\epsi^2\Delta) \epsi^2 \Delta_\Ss v| + C|\trho_2(\epsi^2 \Delta) \epsi (P_\epsi-\lambda) v| + O(\epsi^N) |v| \\ \leq C|\trho_2(\epsi^2 \Delta) \epsi (P_\epsi-\lambda) v| + O(\epsi^N) |v|.
\end{equations}
In the second line we used Theorem \ref{thm:1} with the pair $(\rho_1, \trho_2)$. To show \eqref{eq:3a}, it remains to control $|\trho_2(\epsi^2 \Delta) \epsi (P_\epsi-\lambda) v|$ by $|[\rho_2(\epsi^2 \Delta) \epsi (P_\epsi-\lambda)]_{r,s} v|$. We will need the following lemma:

\begin{lem}\label{lem:8} Let $m \leq 0$ and $B_0 \in \Psi_\epsi^m$ such that $\WF_\epsi(B_0) \subset \Ell_\epsi(\rho_2(\epsi^2\Delta))$. For every $N > 0$ there exists $B_1 \in \Psi_\epsi^{m-1/4}$ with $\WF_h(B_1) \subset \Ell_\epsi(\rho_2(\epsi^2\Delta))$ such that
\begin{equation}\label{eq:0w}
|B_0 \epsi(P_\epsi-\lambda)  v| \leq |B_0\epsi[P_\epsi - \lambda]_{r,s} v| + O(\epsi^{1/3}) |B_1 \epsi(P_\epsi-\lambda) v| + O(\epsi^N) |v|.
\end{equation}
\end{lem}

\begin{proof} The idea is similar to the second part of the proof of Theorem \ref{thm:1}. We have
\begin{equations}\label{eq:7o}
B_0 \epsi(P_{\epsi}-\lambda) = B_0 \epsi[P_{\epsi}-\lambda]_{r,s} + \epsi \sum_{j=1}^n \epsi X_j \cdot  \Psi^{m+}_\epsi + \epsi \Psi^{m+}_\epsi \\
 = B_0 \epsi[P_{\epsi}-\lambda]_{r,s} + \epsi \Lambda_{1/3} \cdot \sum_{j=1}^n \epsi X_j \cdot \Psi^{m-1/4}_\epsi + \epsi \Lambda_{2/3} \cdot \Psi^{m-1/4}_\epsi.
\end{equations}
Let $\trho_3, \trho_4 \in C^\infty(\R^3)$ satisfying \eqref{eq:0g} and such that $\WF_\epsi(B_0) \cap \WF_\epsi(\trho_3(\epsi^2\Delta)-\Id) = \emptyset$. Equivalently, $\trho_3(\epsi^2 \Delta) B_0 = B_0 + \epsi^\infty \Psi^{-\infty}_\epsi$. We multiply both sides of \eqref{eq:7o} by $\trho_3(\epsi^2 \Delta)$ to obtain $B_0 \epsi(P_{\epsi}-\lambda) - B_0 \epsi[P_{\epsi}-\lambda]_{r,s}$
\begin{equation}\label{eq:8c}
= \epsi \Lambda_{1/3} \cdot \trho_3(\epsi^2 \Delta) \sum_{j=1}^n \epsi X_j \cdot \epsi \Psi^{m-1/4}_\epsi + \epsi \Lambda_{2/3} \cdot \trho_3(\epsi^2 \Delta) \cdot \Psi^{m-1/4}_\epsi + \epsi^\infty \Psi^{-\infty}_\epsi.
\end{equation}
Thus there exist operators $\tB_1^j \in \Psi^{m-1/3+}_\epsi \subset \Psi^{m-1/4}_\epsi$ and $\tB_2^j \in \Psi^{m-2/3+} \subset \Psi^{m-1/4}_\epsi$ with wavefront sets contained in $\WF_\epsi(B_0)$ such that $|B_0 \epsi(P_\epsi-\lambda) v - B_0 \epsi [P_{\epsi}-\lambda]_{r,s} v|$
\begin{equation}\label{eq:0c}
\leq \epsi \sum_{j=1}^n \sum_{k=1,2} |\trho_3(\epsi^2 \Delta) \epsi X_j  \tB_k^j v|_{H^{1/3}_\epsi} + \epsi |\trho_3(\epsi^2 \Delta) \tB_k^j v|_{H^{2/3}_\epsi} + O(\epsi^\infty) |v|.
\end{equation}
Theorem \ref{thm:1} applied to $(\trho_3, \trho_4)$ estimates the right hand side of \eqref{eq:0c}:
\begin{equation*}
|B_0 \epsi(P_\epsi-\lambda) v| \leq |B_0 \epsi[P_{\epsi}-\lambda]_{r,s} v| + O(\epsi^{1/3}) \sum_{j,k} |\trho_4(\epsi^2 \Delta) \epsi(P_\epsi-\lambda) \tB_k^j v| + O(\epsi^N)|v|.
\end{equation*}
Since $\WF_\epsi(\tB_k^j) \cap \WF_\epsi(\trho_3(\epsi^2 \Delta) - \Id)$ is empty and $\trho_4=1$ on $\supp(\trho_3)$, $\trho_4(\epsi^2\Delta) \epsi(P_\epsi-\lambda) \tB_k^j = \epsi(P_\epsi-\lambda) \tB_k^j+ \epsi^\infty \Psi^{-\infty}_\epsi$. It follows that
\begin{equation}\label{eq:6k}
|B_0 \epsi(P_\epsi-\lambda) v| \leq |B_0 \epsi[P_{\epsi}-\lambda]_{r,s} v| + O(\epsi^{1/3}) \sum_{j,k} |\epsi(P_\epsi-\lambda) \tB_k^j v| + O(\epsi^N)|v|.
\end{equation}

Fix $1 \leq j \leq n$ and $1 \leq k \leq 2$. We recall that $\tB_k^j \in \Psi^{m-1/4}_\epsi$ with wavefront sets contained in $\WF_\epsi(B_0)$. Similarly to \eqref{eq:8c}, we write
\begin{equation*}
[\epsi (P_\epsi-\lambda), \tB_k^j] = \epsi\Lambda_{1/3} \cdot \trho_3(\epsi^2 \Delta) \sum_{\ell=1}^n \epsi X_\ell  \tB_{k,1}^{j,\ell} + \epsi \Lambda_{2/3} \cdot \trho_3(\epsi^2\Delta) \cdot \epsi \tB_{k,2}^{j,\ell} + \epsi^\infty \Psi_\epsi^{\infty},
\end{equation*}
for some operators $\tB_{k,1}^{j,\ell}, \tB_{k,2}^{j,\ell} \in \Psi^{m-1/2}_\epsi$ with wavefront sets contained in $\WF_h(B_0)$. And similarly to \eqref{eq:6k}, we obtain the estimate
\begin{equation}\label{eq:3d}
|\epsi (P_\epsi-\lambda) \tB_k^j v| \leq |\tB_k^j \epsi (P_\epsi-\lambda) v| + O(\epsi^{1/3}) \sum_{\ell, k, k'} |\epsi(P_\epsi-\lambda) \tB_{k,k'}^{j,\ell} v| + O(\epsi^N)|v|.
\end{equation}
We observe that the terms $O(\epsi^{1/3}) |\epsi(P_\epsi-\lambda) \tB_{k,k'}^{j,\ell} v|$ above involve a factor $\epsi^{1/3}$ and an operator $\tB_{k,k'}^{j,\ell}$ that is $1/4$-smoother than $\tB_k^j$. Since $\epsi (P_\epsi - \lambda) \cdot \Psi_\epsi^{-2} \subset \Psi_\epsi^0$, we can then iterate \eqref{eq:3d} sufficiently many times to get an operator $B_1 \in \Psi^{m-1/4}_h$ with wavefront set contained in $\WF_h(B_0)$, such that
\begin{equation*}
\sum_{j,k} |\epsi (P_\epsi-\lambda) \tB_k^j v| \leq |B_1 \epsi(P_\epsi -\lambda) v| + O(\epsi^N)|v|.
\end{equation*}
We combine this bound with \eqref{eq:6k} to conclude the proof.\end{proof}

The right hand side of \eqref{eq:0w} involves the term $O(\epsi^{1/3}) |B_1 \epsi(P_\epsi-\lambda) v|$ which comes with the factor $\epsi^{1/3}$, and the operator $B_1$. This operator is $1/4$-smoother than $B_0$. We can then iterate \eqref{eq:0w} sufficiently many times starting from $B_0 = \trho_2(\epsi^2 \Delta) \in \Psi_\epsi^{0}$ to obtain operators $B_1 \in \Psi_\epsi^{-1/4}, ..., B_{3N} \in \Psi^{-3N/4}_\epsi$ with wavefront sets contained in $\Ell_\epsi(\rho_2(\epsi^2 \Delta))$ and such that
\begin{equation*}
|\trho_2(\epsi^2 \Delta) \epsi (P_\epsi-\lambda) v| \leq \sum_{k=0}^{3N-1} \epsi^{k/3} |B_k \epsi [P_{\epsi}-\lambda]_{r,s} v| + O(\epsi^N)|\epsi (P_\epsi-\lambda) B_{3N} v|.
\end{equation*}
For $N$ large enough, $\epsi (P_\epsi-\lambda) B_{3N} \in \Psi^0_\epsi$ and $O(\epsi^N)|\epsi (P_\epsi-\lambda) B_{3N} v| = O(\epsi^N)|v|$. In addition the operator $[\rho_2(\epsi^2 \Delta)]_{r,s}$ is elliptic on the wavefront set of the $B_k$ thus
\begin{equation*}
|\trho_2(\epsi^2 \Delta) \epsi (P_\epsi-\lambda) v| \leq |[\rho_2(\epsi^2 \Delta)\epsi (P_\epsi-\lambda)]_{r,s} v| + O(\epsi^N)|v|.
\end{equation*}
Plug this estimate back in \eqref{eq:3e} to conclude the proof of the proposition.\end{proof}

\textit{Starting now we consider $R,N,r,s$ fixed, $\epsi_0$ given by Proposition \ref{prop:1d} and $\epsi, h$ satisfying $0 < \epsi \leq h \leq \epsi_0$.} Fix $\rho_1, \rho_2$ satisfying \eqref{eq:0g}, $\chi_1 \de 1-\rho_1$ and $\chi$ be equal to $1$ near $0$ and such that $\chi \rho_2 = 0$. Define $Q\de \chi(h^2\Delta)$ and
\begin{equations}\label{Eq:2i}
P_{\epsi}(\lambda) \de h(P_\epsi-\lambda)-iQ = -ih\epsi \Delta_\Ss -ih H_1-\lambda h-iQ ,\\
\tP_{\epsi}(\lambda) \de  -i h\epsi \chi_1(\epsi^2 \Delta) \Delta_\Ss -i h H_1-\lambda h-iQ.
\end{equations}
If $P_{\epsi}(\lambda) u \de f$ then $\tP_{\epsi}(\lambda)u = f + ih \epsi\rho_1(\epsi^2 \Delta)\Delta_\Ss u \de  F$. We  use \eqref{eq:0j} to go from the space $H^{r,s}_h$ to the space $H^{r,s}_1$ and we bound $F$ by Proposition~\ref{prop:1d}:
\begin{equations*}
|F|_{H^{r,s}_h} \leq |f|_{H^{r,s}_h} + h \epsi |\rho_1(\epsi^2 \Delta)\Delta_\Ss u|_{H^{r,s}_h} \leq |f|_{H^{r,s}_h} + h^{-|s|-|r|+1} \epsi^{-1} |\rho_1(\epsi^2 \Delta) \epsi^2 \Delta_\Ss u|_{H^{r,s}_1} \\
\leq |f|_{H^{r,s}_1}+C h^{-|s|-|r|+1} \epsi^{-1}|\rho_2(\epsi^2\Delta) \epsi(P_\epsi- \lambda)u|_{H^{r,s}_1}+O(h^{-|s|-|r|}\epsi^N)|u|_{H^{r,s}_1}.
\end{equations*}
We note that $\rho_2(\epsi^2\Delta) Q = 0$ because $\epsi \leq h$, hence
\begin{equations*}
h \epsi^{-1} \rho_2(\epsi^2\Delta) \epsi (P_\epsi-\lambda )u = \rho_2(\epsi^2\Delta) \left(h( P_\epsi-\lambda) - iQ \right)u =\rho_2(\epsi^2\Delta) P_{\epsi}(\lambda)u = \rho_2(\epsi^2 \Delta) f.
\end{equations*}
It follows that
\begin{equation}\label{eq:3g}
|F|_{H^{r,s}_h} \leq |f|_{H^{r,s}_h}+C h^{-|s|-|r|}|\rho_2(\epsi^2 \Delta) f|_{H^{r,s}_1}+O(h^{-|s|-|r|}\epsi^N)|u|_{H^{r,s}_1}.
\end{equation}
The operator $\rho_2(\epsi^2 \Delta)$ is bounded on $H^{r,s}_1$ since $\rho_2(\epsi^2 \Delta) \in \Psi_\epsi^0 \subset \Psi_1^0$ and by \eqref{eq:3b},
\begin{equation*}
e^{G_{r,s}(1)} \rho_2(\epsi^2 \Delta) e^{-G_{r,s}(1)} = \rho_2(\epsi^2 \Delta) + \Psi_1^{-1+} \in \Psi^0_1.
\end{equation*}
Therefore $|\rho_2(\epsi^2 \Delta) f|_{H^{r,s}_1} \leq C|f|_{H^{r,s}_1}$; and $|f|_{H^{r,s}_1}$ is controlled by  $h^{-|s|-|r|} |f|_{H^{r,s}_h}$ because of \eqref{eq:0j}. The estimate \eqref{eq:3g} yields
\begin{equation*}
|F|_{H^{r,s}_h} \leq C h^{-2|s|-2|r|} |f|_{H^{r,s}_h} + O(h^{-2|s|-2|r|}\epsi^N)|u|_{H^{r,s}_h}.
\end{equation*}
Recalling that $f=P_{\epsi}(\lambda) u$ and $F=\tP_{\epsi}(\lambda)u$ we obtain the main result of this section:

\begin{theorem}\label{thm:2} For every $R,N \geq 0$, and $r, s \in \R$ there exist $C_{R,N,r,s} > 0$ and $\epsi_0 >0$ such that if $P_\epsi(\lambda)$ and $\tP_\epsi(\lambda)$ are defined in \eqref{Eq:2i},
\begin{equations*}
\lambda \in \Dd(0,R), \ 0 < \epsi \leq h \leq \epsi_0 \\ \Rightarrow \
|\tP_{\epsi}(\lambda)u|_{H^{r,s}_h} \leq C_{R,N,r,s} h^{-2|s|-2|r|} |P_{\epsi}(\lambda)u|_{H^{r,s}_h} + O(h^{-2|s|-2|r|}\epsi^N)|u|_{H^{r,s}_h}.
\end{equations*}\end{theorem}

\section{Stochastic stability of Pollicott--Ruelle resonances}\label{sec:5}

\subsection{Invertibility of $P_{\epsi}(\lambda)$} Recall that $P_\epsi(\lambda)$ is given by $P_\epsi(\lambda) = h(P_\epsi-\lambda)-iQ$ on $H^{r,s}_h$, and let $D^{r,s}_h$ be its domain on $H^{r,s}_h$:
\begin{equation*}
D^{r,s}_h \de \{ u \in H^{r,s}_h, \ H_1 u \in H^{r,s}_h, \Delta_\Ss u \in H^{r,s}_h \},
\end{equation*}
where $H_1 u, \Delta_\Ss u$ are first seen as distributions. We prove here that the operator $P_\epsi(\lambda)$ is invertible from $D^{r,s}_h$ to $H^{r,s}_h$, provided that $\lambda$ is in a compact set and that $h$ is small enough, $s$ is large enough.

\begin{theorem}\label{thm:5} Let $R > 0$ and $r \in \R$. There exists $s_0 > 0$ such that for every $s \geq s_0$, there exists $h_0 > 0$ with
\begin{equation*}
\epsi \leq h \leq h_0, \ \ |\lambda| \leq R \ \ \Rightarrow \ \  P_\epsi(\lambda) : D^{r,s}_h \rightarrow H^{r,s}_h \text{ is invertible. }
\end{equation*}
\end{theorem}

A necessary step to prove this result is a bound of the form $|u|_{H^{r,s}_h} \leq C_h|P_{\epsi}(\lambda)|_{H^{r,s}_h}$. In view of Theorem \ref{thm:2} applied with $N=2|s|+2|r|+1$ it suffices to show that $|u|_{H^{r,s}_h} \leq Ch^{-1} |\tP_{\epsi}(\lambda)|_{H^{r,s}_h}$ where we recall that $\tP_{\epsi}(\lambda)$ is given by
\begin{equation*}
\tP_{\epsi}(\lambda) =  -i h\epsi \chi_1(\epsi^2 \Delta) \Delta_\Ss -i h H_1-\lambda h-iQ.
\end{equation*}
See $\tP_\epsi(\lambda)$ as a pseudodifferential operator in the semiclassical parameter $h$. Its semiclassical principal symbol is $p_\epsi - i q_\epsi$, where $p_\epsi = \sigma_{H_1}$ and 
\begin{equation*}
q_\epsi(x,\xi) = \chi_1\left(\frac{\epsi^2}{h^2} |\xi|_g^2\right) \frac{\epsi}{h} \sigma_{\Delta_\Ss}(x,\xi)+\chi(|\xi|^2_g).
\end{equation*}
It is clear that $p_\epsi$ belongs to $S^{1}/hS^0$. We claim that $q_\epsi$ also belong to $S^1/hS^0$ or equivalently that
\begin{equation}\label{eq:0f}
\chi_1\left(\frac{\epsi^2}{h^2} |\xi|_g^2\right) \frac{\epsi}{h} \sigma_{\Delta_\Ss}(x,\xi) \in S^1/hS^0.
\end{equation}
Recall that $\Delta_\Ss = - \sum_{j=1}^n X_j^2$, write $\sigma_{X_j}$ for the principal symbol of $\frac{h}{i} X_j$ and note that
\begin{equation*}
q_\epsi(x,\xi) = \sum_{j=1}^n \sigma_{X_j}(x,\xi)\chi_1\left(|\xi'|_g^2\right) \sigma_{X_j}(x,\xi')+\chi(|\xi|_g^2),  \ \ \ \ \ \xi' \de  \epsi h^{-1} \xi.
\end{equation*}
It suffices to show that each term in the above sum belongs to $S^1/hS_0$, thus that $(x,\xi) \mapsto \chi_1(|\xi'|_g^2) \sigma_{X_j}(x,\xi')$ belongs to $S^0/hS^{-1}$. When $|\az|+|\beta|>0$,
\begin{equations*}
\lr{\xi}^{|\beta|} \left|\p_x^\az \p_\xi^\beta \left( \chi_1(|\xi'|_g^2) \sigma_{X_j}(x,\xi') \right)\right| = \lr{\xi}^{|\beta|}(\epsi h^{-1})^{|\beta|} |\p_x^\az \p_{\xi'}^\beta \chi_1(|\xi'|_g^2) \sigma_{X_j}(x,\xi')| \\ \leq \lr{\xi'}^{|\beta|} |\p_x^\az \p_{\xi'}^\beta \chi_1(|\xi'|_g^2) \sigma_{X_j}(x,\xi')| \leq  C_{\az \beta},
\end{equations*}
where in the last inequality we used that $\chi'$ vanishes in a neighborhood of $0$ and that $\chi_1(|\xi'|_g^2) \sigma_{X_j}(x,\xi')$ belongs to $S^0$ as a symbol in $\xi'$. Since for $\az=\beta=0$ there is nothing to prove, we obtain \eqref{eq:0f} and $q_\epsi \in S^1/hS^0$.

Hence the operator $\tP_{\epsi}(\lambda)$ belongs to $\Psi^1_h$. We next compute the principal symbol of the operator $[\tP_\epsi(\lambda)]_{r,s}  \de e^{G_{r,s}(h)} P_\epsi(\lambda) e^{-G_{r,s}(h)}$. We write $p_{\epsi,r,s} - i q_{\epsi,r,s}$ for the principal symbol of $[\tP_\epsi(\lambda)]_{r,s}$, where $p_{\epsi,r,s}, q_{\epsi,r,s}$ are real-valued. The symbol $p_{\epsi,r,s}$ is given by:
\begin{equations}\label{Eq:2za}
p_{\epsi,r,s} = \sigma_{H_1} - \left\{ \sigma_{G_{r,s}}, \chi_1 \left( \dfrac{\epsi^2}{h^2} |\xi|_g^2 \right) \dfrac{\epsi}{h} \sigma_{\Delta_\Ss} \right\} \\
 = \sigma_{H_1} - sh \left\{ m, \chi_1 \left( \dfrac{\epsi^2}{h^2} |\xi|_g^2 \right) \dfrac{\epsi}{h} \sigma_{\Delta_\Ss} \right\}  \rho_0(|\xi|_g^2) \log |\xi|_g \mod hS^0.
\end{equations}
Here we used that $\sigma_{G_{r,s}} = \log(|\xi|_g) \rho_0(|\xi|_g^2) m \mod h S^{-1}$ by \eqref{Eq:3c}, and that $\{\sigma_{\Delta_\Ss}, |\xi|_g^2 \} = 0$ because $\Delta_\Ss$ commutes with $\Delta$, see \eqref{eq:1c}. Since $m$ is homogeneous of degree $0$, we deduce from \eqref{eq:0f} and \eqref{Eq:2za} that
\begin{equation}\label{Eq:2z}
p_{\epsi,r,s} = \sigma_{H_1} + sh \log |\xi|_g \cdot S^0 \mod hS^0.
\end{equation}
Similarly the symbol $q_{\epsi,r,s}$ is given by:
\begin{equations}\label{Eq:2k}
q_{\epsi,r,s} = Q(|\xi|_g^2) + \chi\left( \dfrac{\epsi^2}{h^2} |\xi|_g^2 \right) \dfrac{\epsi}{h} \sigma_{\Delta_\Ss} + h \{ \sigma_{G_{r,s}},\sigma_{H_1}\} \\
 = Q(|\xi|_g^2) + \chi\left( \dfrac{\epsi^2}{h^2} |\xi|_g^2 \right) \dfrac{\epsi}{h} \sigma_{\Delta_\Ss} + sh \{m,\sigma_{H_1}\} \rho_0(|\xi|_g^2) \log |\xi|_g \mod h S^0,
\end{equations}
where we used that $h\rho_0 m \{ \sigma_{H_1}, \log |\xi|_g \} \in hS^0$ and that $h \{ \sigma_{H_1}, \rho_0(|\xi|_g^2) \} \log |\xi|_g \in hS^0$. We remark that since $\{m, \sigma_{H_1}\} \geq 0$, $q_{\epsi,r,s}$ is nonnegative when $s \geq 0$.

The key step to prove Theorem \ref{thm:5} is the following Proposition, whose proof is largely inspired from \cite[Proposition 3.1]{DZ1} and \cite[Lemma 4.2]{DZ2}:

\begin{proposition}\label{prop:3} Let $R > 0$, $r \in \R$. There exists $s_0$ such that for $s \geq s_0$, there exist $h_0 > 0$ and $C_{R,r,s} > 0$ with
\begin{equation*}
0 < \epsi \leq h \leq h_0, \ |\lambda| \leq R \ \ \Rightarrow \ \  |u|_{H^{r,s}_h} \leq C_{R,r,s} h^{-1} |\tP_\epsi(\lambda) u|_{H^{r,s}_h}.
\end{equation*}\end{proposition}

\begin{proof} We define $v \de e^{G_{r,s}(h)} u \in L^2$ and we recall that $[A]_{r,s} \de e^{G_{r,s}(h)} A e^{-G_{r,s}(h)}$ when $A \in \Psi^m_h$.  Using a microlocal partition of unity it is sufficient to show the inequality
\begin{equation*}
|[A]_{r,s} v| \leq C h^{-1} |[\tP_\epsi(\lambda)]_{r,s} v| + O(h^\infty)|v|,
\end{equation*}
when $\WF_h(A)$ is supported in a small neighborhood of $(x_0,\xi_0) \in \oT^*S^*\Mm$ in each of the following cases:

\textbf{Case I: $(x_0,\xi_0) \in \Ell_h(Q)$.} Since $\{m,\sigma_{H_1}\} \geq 0$ by construction of $m$, \eqref{Eq:2k} shows that $q_{\epsi,r,s}(x_0,\xi_0) > 0$ when $s \geq 0$. In particular, $[\tP_\epsi(\lambda)]_{r,s}$ is elliptic at $(x_0,\xi_0)$. By the elliptic estimate, $|A_{r,s}v| \leq C |[\tP_\epsi(\lambda)]_{r,s} v| + O(h^\infty) |v|$.

\textbf{Case II: $(x_0,\xi_0) \in \kappa(E_s^*)$.} Here $\kappa : T^*S^*\Mm \rightarrow \p\oT^*S^*\Mm$ is the projection map defined in \cite[Appendix E.1]{DZ}. The operator $\tP_\epsi(\lambda)$ has semiclassical principal symbol $p_\epsi-iq_\epsi$. We note that $q_\epsi \geq 0$ everywhere and that $p_\epsi = \sigma_{H_1}$ is homogeneous of degree $1$ and independent of $h$. Hence we can apply the radial source estimate \cite[Proposition 2.6]{DZ1}. Fix $B_1 \in \Psi_h^0$ with wavefront set contained in the set $\{ \rho_0 m = 1\}$ so that on $WF_h(B_1)$ the space $H^{r,s}_h$ and $H^{r+s}_h$ are microlocally equivalent, see \eqref{Eq:6d}. There exist $s_0 > 0$ and $U_0$ neighborhood of $\kappa(E_s^*)$ in $\oT^*S^*\Mm$ such that  
\begin{equation*}
s \geq s_0, \ \ \WF_h(A) \subset U_0 \ \Rightarrow \ |Au|_{H^{r+s}_h} \leq Ch^{-1}|B_1 \tP_\epsi(\lambda) u|_{H^{r+s}_h} + O(h^\infty)|u|_{H^{-|r|-s}_h}. 
\end{equation*}
After possibly shrinking the size of $\WF_h(A)$ we can use that $H^{r,s}_h$ and $H^s_h$ are microlocally equivalent near $\WF_h(A)$, $\WF_h(B_1)$ to conclude that
\begin{equation*}
|Au|_{H^{r,s}_h} \leq Ch^{-1}|\tP_\epsi(\lambda) v|_{H^{r,s}_h} + O(h^\infty)|u|_{H^{-|r|-s}_h}.
\end{equation*}
Since $H^{r,s}_h$ embeds in $H^{-|r|-s}_h$, we deduce that for $v \de e^{G_{r,s}(h)} u$,
\begin{equation*}
|[A]_{r,s}v| \leq Ch^{-1}|[\tP_\epsi(\lambda)]_{r,s} v| + O(h^\infty)|v|.
\end{equation*}

\textbf{Case III: $(x_0,\xi_0) \in \oT^*S^*\Mm$, $(x_0,\xi_0) \notin \overline{E_0^*} \oplus \overline{E_u^*}$.} In this case $(x_0,\xi_0)$ admits a neighborhood $U$ in $\oT^*S^*\Mm$ such that 
\begin{equation*}
\mathrm{d}( \exp(-tH_{\sigma_{H_1}})(U), \kappa(E_s^*)) \rightarrow 0 \text{ as } t \rightarrow -\infty,
\end{equation*}
see \cite[Equation (3.2)]{DZ1}. Hence for $T$ large enough, $\exp(-TH_{\sigma_{H_1}})(U) \subset U_0$ where $U_0$ is the open set defined in Case II. We recall that $p_{\epsi,r,s} - i q_{\epsi,r,s}$ is the principal symbol of $[\tP_\epsi(\lambda)]_{r,s}$, and that $p_{\epsi,r,s} =  \sigma_{H_1} + h S^{1/2}$ and $q_{\epsi,r,s} \geq 0$. Since
$\sigma_{H_1}$ is homogeneous of degree $1$ we can apply \cite[Proposition 2.2]{DZ2}. It shows that if $B \in \Psi_h^0$ has wavefront set contained in $U_0$ then $|[A]_{r,s} v| \leq C |Bv| + Ch^{-1} |[\tP_\epsi(\lambda)]_{r,s} v| + O(h^\infty)|v|$. Combined with the result of Case II, we get
\begin{equation*}
|[A]_{r,s} v| \leq Ch^{-1}|[\tP_\epsi(\lambda)]_{r,s} v| + O(h^\infty)|v|.
\end{equation*}

\textbf{Case IV: $(x_0,\xi_0) \in E_u^* \setminus 0$.} We recall that the lifted geodesic flow $\exp(-tH_{\sigma_{H_1}})(x_0,\xi_0)$ is equal to $\left(e^{-tH_1}(x_0), ^Tde^{-tH_1}(x_0)^{-1}\xi_0\right)$.  We observe that $\exp(-tH_{\sigma_{H_1}})(x_0,\xi_0)$ converges to the zero section as $t \rightarrow +\infty$: because of $\xi_0 \in E_u^*(x_0) = E_s(x_0)$ and of~\eqref{Eq:7k},
\begin{equations*}
\left|^Tde^{-tH_1}(x_0)^{-1} \xi_0\right|_g = \left|^Tde^{tH_1}(e^{-tH_1}(x_0)) \xi_0\right|_g \leq Ce^{-ct}.
\end{equations*}
Since $\Ell_h(Q)$ contains the zero section, there exists $T > 0$ such that $\exp(-TH_{\sigma_{H_1}})(x_0,\xi_0)$ belongs to $\Ell_h(Q)$. We apply again \cite[Proposition 2.2]{DZ2}: if $\WF_h(A)$ is supported sufficiently close to $E_u^*$, there exists $B \in \Psi_h^0$ with wavefront set contained in the elliptic set of $Q$ such that $|[A]_{r,s} v| \leq C |Bv| + Ch^{-1} |[\tP_\epsi(\lambda)]_{r,s} v| + O(h^\infty)|v|$. Together with Case I, it implies
\begin{equation*}
|[A]_{r,s} v| \leq Ch^{-1} |[\tP_\epsi(\lambda)]_{r,s} v| + O(h^\infty)|v|.
\end{equation*}

\textbf{Case V: $(x_0,\xi_0) \in \kappa(E_u^*)$.} We recall that $q_\epsi \geq 0$ everywhere and that $p_\epsi = \sigma_{H_1}$ is homogeneous of degree $1$ and independent of $h$.  Hence we can apply \cite[Proposition 2.7]{DZ1}. Fix $B_1 \in \Psi_h^0$ elliptic on $\kappa(E_u^*)$, such that $\WF_h(B_1) \cap \overline{E_0^*} = \emptyset$ and such that $\rho_0 m = -1$ on $\WF_h(B_1)$. Then (after possibly increasing the value of $s_0$ given in Case II) there exist a neighborhood $U_1$ of $(x_0,\xi_0)$ and $B \in \Psi_h^0$ with $\WF_h(B) \subset \WF_h(B_1) \setminus \kappa(E_u^*)$ such that if $\WF_h(A) \subset U_1$ and $s \geq s_0$,
\begin{equation}\label{Eq:2l}
|Au|_{H^{r-s}_h} \leq C|Bu|_{H^{r-s}_h}+ C h^{-1}|B_1 \tP_\epsi(\lambda)u|_{H^{r-s}_h} + O(h^\infty)|u|_{H^{-|r|-s}_h}.
\end{equation}
Without loss of generality $U_1$ is small enough so that the spaces $H^{r-s}_h$, $H^{r,s}_h$ are microlocally equivalent on $\WF_h(A), \WF_h(B_1), \WF_h(B)$. Hence we can replace $H_h^{r-s}$ by $H^{r,s}_h$ in \eqref{Eq:2l}. In addition since $\WF_h(B_1)$ is supported away from $\kappa(E_0^*)$, it can be written as a finite sum of operators in $\Psi_h^0$ whose wavefront sets are supported near points $(x_0,\xi_0)$ satisfying Cases I-IV. Finally, since $H_h^{r,s}$ embedds in $H^{-s-|r|}_h$, the term $O(h^\infty)|u|_{H^{-|r|-s}_h}$ in the right hand side of \eqref{Eq:2l} is bounded by $O(h^\infty)|u|_{H^{r,s}_h}$. It follows that 
\begin{equation*}
|Au|_{H^{r,s}_h} \leq Ch^{-1}|\tP_\epsi(\lambda) u|_{H^{r,s}_h} + O(h^\infty)|u|_{H^{r,s}_h}.
\end{equation*}
Since $v = e^{G_{r,s}(h)} u$ we deduce that
\begin{equation*}
|[A]_{r,s}v| \leq Ch^{-1}|[\tP_\epsi(\lambda)]_{r,s} v| + O(h^\infty)|v|.
\end{equation*}

\textbf{Case VI: $(x_0,\xi_0) \in \overline{E_0^*} \setminus \Ell_h(Q)$.} In particular, $\xi_0 \neq 0$ and $\sigma_{H_1}(x_0,\xi_0) \neq 0$. By \eqref{Eq:2z}, we have $p_{\epsi,r,s} = \sigma_{H_1} + sh\log|\xi|_g \cdot S^0$. 
This shows that the operator $[\tP_\epsi(\lambda)]_{r,s}$ is elliptic at $(x_0,\xi_0)$. Therefore if $A$ has wavefront set contained in a small neighborhood of $(x_0,\xi_0)$ the elliptic estimate shows that $|[A]_{r,s}v| \leq C |[\tP_\epsi(\lambda)]_{r,s} v| + O(h^\infty)|v|$.

Since Cases I-VI cover the whole $\oT^*S^*\Mm$ this ends the proof of the theorem.
\end{proof}

\begin{proof}[Proof of Theorem \ref{thm:5}.] It is very similar to the end of the proof of \cite[Proposition 3.1]{DZ1}. Fix $R > 0$ and $r \in \R$. Proposition \ref{prop:3} shows that $|u|_{H^{r,s}_h} \leq C_R h^{-1} |\tP_\epsi(\lambda) u|_{H^{r,s}_h}$ as long as $0 < \epsi \leq h \leq h_0$ and $s$ is large enough. Theorem \ref{thm:2} applied with $N=2|s|+2|r|+1$ yields the estimate
\begin{equation*}
|u|_{H^{r,s}_h} \leq C_{R} h^{-2r-2s-1} |P_{\epsi}(\lambda)u|_{H^{r,s}_h} + O(h)|u|_{H^{r,s}_h}.
\end{equation*}
After possibly decreasing the value of $h_0$ we can absorb the term $O(h)|u|_{H^{r,s}_h}$ by the left hand side. We get $|u|_{H^{r,s}_h} \leq C_{R} h^{-2r-2s-1} |P_{\epsi}(\lambda)u|_{H^{r,s}_h}$. This estimate implies that the operator $P_{\epsi}(\lambda) : D^{r,s}_h \rightarrow H^{r,s}_h$ is injective.

To show the surjectivity of $P_\epsi(\lambda)$ we first note that the range of $P_{\epsi}(\lambda)$ is closed in $H^{r,s}_h$. Indeed, let $u_j \in D^{r,s}_h$ such that $P_{\epsi}(\lambda)u_j$ converges in $H^{r,s}_h$. Then $u_j$ is a Cauchy sequence in $H^{r,s}_h$ and it converges to some $u \in H^{r,s}_h$. We must show that $u \in D^{r,s}_h$. The sequence $P_{\epsi}(\lambda)u_j$ is bounded in $H^{r,s}_h$ hence it converges weakly; it follows that  $P_{\epsi}(\lambda)u \in H^{r,s}_h$. By Proposition \ref{prop:1d}, $\rho_1(\epsi^2\Delta) \Delta_\Ss u \in H^{r,s}_h$. In addition for any $\epsi > 0$, $\chi_1(\epsi^2 \Delta) \Delta_\Ss u \in C^\infty$. It follows that $\Delta_\Ss u \in H^{r,s}_h$ hence $H_1 u \in H^{r,s}_h$. Therefore $u$ belongs to the domain of $P_\epsi(\lambda)$ and the range of $P_\epsi(\lambda)$ is closed. 

To conclude we show that the range of $P_\epsi(\lambda)$ is dense in $H_h^{r,s}$. The dual of $H^{r,s}_h$ is $H^{-r,-s}_h$. Thus it suffices to prove that if $f \in H^{-r,-s}_h$ is such that $\lr{f,P_\epsi(\lambda)u} = 0$ for every $u \in H^{r,s}_h$ then $f=0$, or equivalently that $P_\epsi(\lambda)$ is injective. We have
\begin{equation*}
P_{\epsi}(\lambda) = -ih\epsi \Delta_\Ss -ih H_1-\lambda h-iQ, \ \ \ 
-P_\epsi(-\olambda)^* = -ih\epsi \Delta_\Ss +ih H_1-\lambda h-iQ. 
\end{equation*}
Therefore $-P_\epsi(-\olambda)^*$ is equal to $P_\epsi(\lambda)$ except for $H_1$ which is replaced by $-H_1$. For the dynamics of $-H_1$, $E_u^*$ is a radial source and $E_s^*$ a radial sink. Moreover the imaginary part of $-P_\epsi(\lambda)^*$ is non-positive. The space $H^{-r,-s}_h$ has low regularity near $E_s^*$ (the radial sink for $-H_1$) since it is microlocally equivalent to $H^{-r-s}_h$ near $E_s^*$. Similarly $H^{-r,-s}_h$ has high regularity near $E_s^*$ (the radial source for $-H_1$) since it is microlocally equivalent to $H^{-r-s}_h$ near $E_s^*$. Hence the same analysis as in the proof of Proposition \ref{prop:3} can be applied to $-P_\epsi(\lambda)^*$. It shows that for $s$ large enough and $0 < \epsi \leq h$ small enough, $\lambda \in \Dd(0,R)$,
\begin{equation*}
|f|_{H_h^{-r,-s}} \leq C_R h^{-2r-2s-1}|P_\epsi(-\olambda)^* f|_{H_h^{-r,-s}}.
\end{equation*}
This shows that $P_\epsi(\lambda)^*$ is injective. Hence the range of $P_\epsi(\lambda)$ is dense and $P_\epsi(\lambda)$ is surjective. This ends the proof of the theorem.\end{proof}

\subsection{Proof of Theorem \ref{thm:0}}\label{subsec:5.2}

We conclude the paper with a more precise version of Theorem \ref{thm:0}. A function $\epsi \in [0,\epsi_0) \mapsto f(\epsi)$ is said to be $C^1([0,\epsi_0))$ if $f$ is $C^1$ on $(0,\epsi_0)$ and $f'(\epsi)$ has a limit when $\epsi \rightarrow 0$. By induction we define the class $C^k([0,\epsi_0))$. In the following, we shall say that $f$ is smooth at $0$ if for every $k > 0$, there exists $\epsi_k > 0$ such that $f \in C^k([0,\epsi_k))$. The set $\Sigma(P_\epsi)$ (resp. $\Res(P_0)$) is defined as the $L^2$-spectrum of $P_\epsi = \frac{1}{i}(H_1+\epsi \Delta_\Ss)$ (resp. Pollicott--Ruelle resonances of $P_0 = \frac{1}{i}H_1$), with inclusion according to multiplicity. 

\begin{theorem}\label{thm:6} The set of accumulation points of $\Sigma(P_\epsi)$, as $\epsi \rightarrow 0$, is contained in $\Res(P_0)$. Conversely, if $\lambda_0 \in \Res(P_0)$ has multiplicity $m$, there exist $r_0 > 0, \epsi_0 > 0$ such that for every $\epsi \in (0,\epsi_0)$, $\Sigma(P_\epsi) \cap \Dd(\lambda_0,r_0) = \{\lambda_j(\epsi)\}_{j=1}^m$. 
Moreover,
\begin{enumerate}
\item[$(i)$] If $m=1$, then $\epsi \mapsto \lambda_1(\epsi)$ is smooth at $\epsi = 0$ and
\begin{equation}\label{eq:0m}
\lambda_1(\epsi) = \lambda_0 + i \epsi \int_{S^*\Mm} \lr{\nabla_\Ss u, \nabla_\Ss v} d\mu + O(\epsi^2),
\end{equation}
where $u, v$ are the left and right resonant states defined in Lemma \ref{lem:3}.
\item[$(ii)$] The finite-rank operators 
\begin{equation}\label{eq:0h}
\Pi_\epsi \de \dfrac{1}{2\pi i} \oint_{\p \Dd(\lambda_0,r_0)} (P_\epsi-\lambda)^{-1} d\lambda \ : \  C^\infty(S^*\Mm) \rightarrow \DD'(S^*\Mm)
\end{equation}
form a smooth trace-class family of operators at $\epsi=0$.
\end{enumerate}
\end{theorem}

\begin{rmk}\label{rem:1} Theorem \ref{thm:6} shows that as $\epsi \rightarrow 0^-$, the spectrum of $P_{-\epsi}^*$ converges to complex conjugates of Pollicott--Ruelle resonances. Because of the identity $P_\epsi = P_{-\epsi}^*$, we deduce that the spectrum of $P_\epsi$ converges to complex conjugates of Pollicott--Ruelle resonances as $\epsi \rightarrow 0^-$. 
\end{rmk}

\begin{proof} Fix $R > 0$ and $k_0$ a positive integer. For $1 \leq k \leq k_0$, let $r_k \de 2^{k+1}-2$. By Theorem \ref{thm:5} and \cite[Proposition 3.4]{DZ1} there are $s_0, h_0 > 0$ such that for every $0 \leq \epsi \leq h_0$, $r \in [\![0,r_{k_0}]\!]$ and $\lambda \in \Dd(0,R)$ the operator
\begin{equation*}
P_\epsi(\lambda) = -i\epsi h_0 \Delta_\Ss -ih_0 H_1 - \lambda h_0 - iQ
\end{equation*}
admits a right inverse on $\HH^{-r} \de H_{h_0}^{-r,s_0}$: there exists a bounded operator $P_\epsi(\lambda)^{-1} : \HH^{-r} \rightarrow \HH^{-r}$ with range contained in the domain of $P_\epsi(\lambda)$ such that $P_\epsi(\lambda) P_\epsi(\lambda)^{-1} = \Id_{\HH^{-r}}$. We show below that for every $r \in [\![0, r_{k_0}-r_k]\!]$, the operator $P_\epsi(\lambda)^{-1} : \HH^{-r} \rightarrow \HH^{-r-r_k}$ is $C^k([\![0,h_0))$. We proceed by induction on $k$.

We start with $k=1$. For every $r \in [\![0, r_{k_0}]\!] \cap [\![-2, r_{k_0}-2]\!] = [\![0,r_{k_0}-r_1]\!]$, the operator $P_\epsi(\lambda)^{-1}$ maps $\HH^{-r}$ to itself and $\HH^{-r-2}$ to itself. This fact, together with the identity
\begin{equation}\label{eq:0n}
P_{\epsi}(\lambda)^{-1} - P_{\epsi'}(\lambda)^{-1} = -i (\epsi-\epsi') P_{\epsi'}(\lambda)^{-1} h_0 \Delta_\Ss P_{\epsi}(\lambda)^{-1}
\end{equation}
shows that $\epsi \in [0,h_0) \mapsto P_{\epsi}(\lambda)^{-1} : \HH^{-r} \rightarrow \HH^{-r-2}$ is differentiable (in particular continuous) with
\begin{equation}\label{eq:0l}
\p_\epsi P_\epsi(\lambda) =  -i P_{\epsi}(\lambda)^{-1} h_0 \Delta_\Ss P_{\epsi}(\lambda)^{-1}.
\end{equation}
The right hand side of \eqref{eq:0l} is continuous, hence $\epsi \in [0,h_0) \mapsto P_{\epsi}(\lambda)^{-1} : \HH^{-r} \rightarrow \HH^{-r-2}$ is $C^1([0,h_0))$.

Assume now that $k \leq {k_0}-1$ and that for every $r \in [\![0, r_{k_0}-r_k]\!]$, $P_\epsi(\lambda)^{-1} : \HH^{-r} \rightarrow \HH^{-r-r_k}$ is $C^k([\![0,h_0))$. The identity \eqref{eq:0l} shows that $\p_\epsi P_\epsi(\lambda) : \HH^{-r} \rightarrow \HH^{-r-2r_k-2}$ is also $C^k([0,h_0))$ as long as $r \in [\![0,r_{k_0}-r_k]\!] \cap [\![-r_k - 2, r_{k_0}-2r_k-2]\!]$. Since $r_{k+1} = 2r_k-2$, the operator $\p_\epsi P_\epsi(\lambda) : \HH^{-r} \rightarrow \HH^{-r-r_{k+1}}$ is $C^k([\![0,h_0))$ as long as $r \in [0, r_{k_0}-r_{k+1}]\!]$. This implies that $P_\epsi(\lambda) : \HH^{-r} \rightarrow \HH^{-r-r_{k+1}}$ is $C^{k+1}([0,h_0))$ for the above range of $r$. This completes the induction process.

It follows that the operator $P_\epsi(\lambda)^{-1} : \HH^0 \rightarrow \HH^{-r_{k_0}}$ is $C^{k_0}([0,h_0))$. We recall that $Q$ is a smoothing operator. In particular, $Q$ maps $\HH^{-r_{k_0}}$ to the Sobolev space $H^N$ for any $N$. It follows that $Q P_\epsi(\lambda)$ is a trace-class operator with holomorphic dependence in $\lambda \in \Dd(0,R)$ and $C^{k_0}$ dependence in $\epsi \in [0,h_0)$. Since ${k_0}$ was arbitrary, $Q P_\epsi(\lambda)$ is smooth at $\epsi = 0$. For $\epsi \in [0,h_0)$ and $\lambda \in \Dd(0,R)$, we define the Fredholm determinant 
\begin{equation*}
D_\epsi(\lambda) = \Det_{\HH^0}(\Id + i Q P_\epsi(\lambda)^{-1}),
\end{equation*}
which depends holomorphically in $\lambda$, and which is smooth at $\epsi=0$.

The operator $h_0(P_\epsi - \lambda) = P_\epsi(\lambda) + iQ$ is Fredholm, because where $P_\epsi(\lambda)$ admits a right inverse on $\HH^0$ and $Q$ is compact. Hence, the $\HH^0$-spectrum of $P_\epsi$ in $\Dd(0,R)$ is discrete and equal to the zero set of $D_\epsi(\lambda)$. When $\epsi \neq 0$ the operator $P_\epsi$ is subelliptic. Consequently, $\HH^0$-eigenvectors of $P_\epsi$ must belong to the (standard) Sobolev space $H^2$, thus to the domain of $P_\epsi$ on $L^2$. Conversely, $L^2$-eigenvectors of $P_\epsi$ must belong to the (standard) Sobolev space $H^{s_0}$, thus to $\HH^0$. This shows that for $\epsi \neq 0$, the $L^2$-spectrum and $\HH^0$-spectrum of $P_\epsi$ in $\Dd(0,R)$ are equal, and the $L^2$-eigenvalues of $P_\epsi$  in $\Dd(0,R)$ are exactly the~zeroes~of~$D_\epsi(\lambda)$.

For $\epsi > 0$, $D_\epsi(\lambda)$ is a holomorphic function of $\lambda$ whose zero set is the $L^2$-spectrum of $P_\epsi$ in $\Dd(0,R)$, and the zero set of $D_0(\lambda)$ is the Pollicott--Ruelle spectrum of $P_0$ in $\Dd(0,R)$ -- see \cite[Proposition 3.2]{DZ}. Since $D_\epsi(\lambda)$ is smooth at $\epsi = 0$, the first part of the theorem follows from an application of Hurwitz's theorem.  

If $\lambda_0$ is a Pollicott--Ruelle resonance of $P_0$ and $\lambda_1(\epsi)$ is the unique eigenvalue of $P_\epsi$ converging to $\lambda_0$, the implicit function theorem shows that $\epsi \mapsto \lambda_1(\epsi)$ is smooth. We compute now the leading terms in the expansion \eqref{eq:0m}, inspired by the method of \cite[\S 3.1]{D}. Denote by $\Res(P_0)$ the set of Pollicott--Ruelle resonances of $P_0 = \frac{1}{i}H_1$ and fix $K$ be a compact subset of $\Dd(0,R) \setminus \Res(P_0)$. For every $\lambda \in K$, $D_0(\lambda) \neq 0$ and the operator $\Id + iQ P_0(\lambda)^{-1} : \HH^0 \rightarrow \HH^0$ is invertible. Therefore, for every $0 < \epsi \leq h_0$ and $\lambda \in K$,
\begin{equation*}
\Id + iQ P_\epsi(\lambda)^{-1} = \left(\Id + iQP_0(\lambda)^{-1}\right) \cdot \left( \Id + \left( \Id + iQP_0(\lambda)^{-1}\right)^{-1} iQ \left(P_\epsi(\lambda)^{-1} - P_0(\lambda)^{-1}\right) \right).
\end{equation*}
Uniformly for $\lambda \in K$, the operator $\left( \Id + iQP_0(\lambda)^{-1} \right)^{-1}$ is bounded on $\HH^0$ and by \eqref{eq:0n}, $Q \left(P_\epsi(\lambda)^{-1} - P_0(\lambda)^{-1} \right)$ has trace-class norm $O(\epsi)$. The identity \eqref{eq:0n} implies for $\lambda \in K$,
\begin{equations*}
D_\epsi(\lambda) = D_0(\lambda) \cdot \Det_{\HH^0} \left( \Id + i\left( \Id + iQP_0(\lambda)^{-1}\right)^{-1} Q \left(P_\epsi(\lambda)^{-1} - P_0(\lambda)^{-1}\right) \right) \\
 = D_0(\lambda) \cdot \left( 1 + \epsi h_0 \Trace_{\HH^0}\left(\left( \Id + iQP_0(\lambda)^{-1}\right)^{-1}  Q  P_0(\lambda)^{-1} \Delta_\Ss P_{\epsi}(\lambda)^{-1} \right) + O(\epsi^2) \right).
 \end{equations*}
The operator $Q P_0(\lambda)^{-1} \Delta_\Ss$ extends to a trace-class operator in $\HH^0$. Because of the identity \eqref{eq:0n}, we have uniformly for $\lambda \in K$,
\begin{equations*}
h_0 \Trace_{\HH^0}\left(\left( \Id + iQP_0(\lambda)^{-1}\right)^{-1}  Q  P_0(\lambda)^{-1} \Delta_\Ss P_{\epsi}(\lambda)^{-1} \right) = f_1(\lambda) + O(\epsi), \\ 
f_1(\lambda) \de h_0 \Trace_{\HH^0}\left(\left( \Id + iQP_0(\lambda)^{-1}\right)^{-1}  Q  P_0(\lambda)^{-1} \Delta_\Ss P_0(\lambda)^{-1} \right).
\end{equations*}
It follows that uniformly for $\lambda \in K$, 
\begin{equation}\label{eq:0i}
D_\epsi(\lambda) = D_0(\lambda) \cdot \left( 1 + f_1(\lambda) \epsi + O(\epsi^2) \right).
\end{equation}

In \eqref{eq:0i}, the function $D_\epsi$ is holomorphic on $\Dd(0,R)$ and $f_1(\lambda)$ is meromorphic in $\Dd(0,R)$, with poles in $\Dd(0,R) \cap \Res(P_0)$. Therefore we can apply \cite[Lemma 4.4]{D} with $E = \Dd(0,R)$, $S_0 = \Res(P_0)$, $D_\epsi(\lambda)/D_0(\lambda) = 1 + f_1(\lambda) \epsi + O(\epsi^2)$ and $g(\lambda,\epsi) = D_0(\lambda)$ (strictly speaking, \cite[Lemma 4.4]{D} is stated there with $E = \C$ or $\C \setminus 0$; but it also holds without change in the proof when $E = \Dd(0,R)$). It shows that \eqref{eq:0i} is valid uniformly for $\lambda \in \Dd(0,R) \setminus \Res(P_0)$ and that the function $D_0(\lambda) f_1(\lambda)$ is holomorphic on $\Dd(0,R)$.

Let $\lambda_0 \in \Dd(0,R)$ be a simple resonance of $\Res(P_0)$. We now work with $f_1(\lambda)$ for $\lambda$ in a small punctured disk $\DD \setminus \lambda_0 \subset \Dd(0,R)$, so that $\lambda_0$ is the only resonance of $P_0$ in $\DD$. We have
\begin{equations*}
f_1(\lambda) = h_0 \Trace_{\HH^0}\left(\left( \Id + iQP_0(\lambda)^{-1}\right)^{-1}  Q  P_0(\lambda)^{-1} \Delta_\Ss P_0(\lambda)^{-1} \right) \\
 = h_0 \Trace_{\HH^0}\left(P_0(\lambda)^{-1} \left( \Id + iQP_0(\lambda)^{-1}\right)^{-1}  Q  P_0(\lambda)^{-1} \Delta_\Ss  \right) = \Trace_{\HH^0}\left( (P_0-\lambda)^{-1}  Q  P_0(\lambda)^{-1} \Delta_\Ss  \right).
\end{equations*}
In the above we used the cyclicity of the trace and the identity
\begin{equation*}
P_0(\lambda)^{-1} \left( \Id + iQP_0(\lambda)^{-1}  \right)^{-1} = (P_0(\lambda) + iQ)^{-1} = h_0^{-1} (P_0 - \lambda )^{-1}.
\end{equation*}

Because of \eqref{eq:0k} and since $P_0(\lambda)^{-1}$ is holomorphic near $\lambda_0$, we can write
\begin{equation}\label{eq:0ya}
(P_0-\lambda)^{-1}  Q  P_0(\lambda)^{-1} \Delta_\Ss = (i(P_0-\lambda)^{-1} - h_0 P_0(\lambda)^{-1}) \Delta_\Ss = \dfrac{i u \otimes v \Delta_\Ss}{\lambda-\lambda_0} + B(\lambda),
\end{equation}
where $B(\lambda)$ denotes a holomorphic family of operators near $\lambda_0$. The right hand side of \eqref{eq:0ya} is  trace-class  on $\HH^0$ and the operator $u \otimes v \Delta_\Ss$ is of rank $1$. Therefore $B(\lambda)$ is trace-class on $\HH^0$ and $f_0(\lambda) \de \Trace_{\HH^0}(B(\lambda))$ is holomorphic. It follows that
\begin{equation*}
f_1(\lambda)-f_0(\lambda) = \dfrac{i\Trace_{\HH^0}\left(u \otimes v \Delta_\Ss \right)}{\lambda-\lambda_0} = \dfrac{i \Trace_{\HH^0}\left(\Delta_\Ss u \otimes v  \right)}{\lambda-\lambda_0} = \dfrac{i}{\lambda_0-\lambda}\int_{S^*\Mm} \lr{\nabla_\Ss u, \nabla_\Ss v}.
\end{equation*}
In the last equality we used that $\Delta_\Ss u$ and $v$ have wavefront sets contained in $E_u^*$ and $E_s^*$, respectively. Hence the trace of the operator $\Delta_\Ss u \otimes v$ is given by integrating the kernel $\Delta_\Ss u(x) v(y)$ along the diagonal $\{x=y\}$ according to \cite[Proposition 7.6]{GS}. The operator $\nabla_\Ss$ was defined in \S \ref{sec:2.1} and the scalar product $\lr{\cdot, \cdot}$ is inherited from the Euclidean structure on the fibers of $T^*\Mm$.

Combining the above, we obtain that uniformly in $\epsi$ small enough and $\lambda \in \DD \setminus \lambda_0$, 
\begin{equations*}
D_\epsi(\lambda) = D_0(\lambda) - i \epsi \dfrac{D_0(\lambda)}{\lambda-\lambda_0} \int_{S^*\Mm} \lr{\nabla_\Ss u, \nabla_\Ss v} + \epsi D_0(\lambda) f_0(\lambda) + O(\epsi^2) \\
 = D'_0(\lambda_0) \left( \lambda-\lambda_0 - i \epsi \int_{S^*\Mm} \lr{\nabla_\Ss u, \nabla_\Ss v} + O(\epsi(\lambda-\lambda_0)) + O(\epsi^2) \right).
\end{equations*}
Recall that $\lambda_1(\epsi)$ is the unique eigenvalue of $P_\epsi$ near $\lambda_0$. In particular $D_\epsi(\lambda_1(\epsi)) = 0$. Since $\epsi \mapsto \lambda_1(\epsi)$ is smooth, $\lambda_1(\epsi) = \lambda_0 + O(\epsi)$. This yields
\begin{equation*}
\lambda_1(\epsi) = \lambda_0 + i \epsi \int_{S^*\Mm} \lr{\nabla_\Ss u, \nabla_\Ss v} + O(\epsi^2).
\end{equation*}
This concludes the proof of $(i)$.

For $(ii)$, we fix $k_0 > 0$ and we recall that $P_\epsi(\lambda)^{-1} : \HH^0 \rightarrow \HH^{-r_{k_0}}$ is $C^{k_0}([0,h_0))$. Since $h_0(P_\epsi-\lambda) = P_\epsi(\lambda)+Q$, where $Q$ is smoothing, the 	family $P_\epsi-\lambda : \HH^{-r_{k_0}} \rightarrow \HH^0$ is Fredholm with $C^{k_0}$ dependence in $\epsi$. Hence, $(P_\epsi-\lambda)^{-1}$ is a meromorphic family of operators with poles of finite rank, with $C^{k_0}$ dependence in $\epsi$. This shows that the family of operators $\epsi \rightarrow \Pi_\epsi : \HH^0 \rightarrow \HH^{-r_{k_0}}$ given by \eqref{eq:0h} is $C^{k_0}([0,h_0))$. A fortiori, $\epsi \mapsto \Pi_\epsi : C^\infty(S^*\Mm) \rightarrow \DD'(S^*\Mm)$ is also $C^{k_0}([0,h_0))$, hence smooth at $\epsi = 0$. \end{proof}


\begin{thebibliography}{2}
\bibitem[ABT15]{ABT} J\"urgen Angst, Isma\"el Bailleul and Camille Tardif, {\em Kinetic Brownian motion on Riemannian manifolds.} Electron. J. Probab. 20(2015).
\bibitem[Ba05]{Ba} Viviane Baladi, \textit{Anisotropic Sobolev spaces and dynamical transfer operators: $C^\infty$ foliations, Algebraic and topological dynamics.} Contemp. Math. 385(2005), 123-135.
\bibitem[BaTs07]{BaTs} Viviane Baladi and Masato Tsujii, \textit{Anisotropic H\"older and Sobolev spaces for hyperbolic diffeomorphisms.} Ann. Inst. Fourier 57(2007), 127-154.
\bibitem[BaTa16]{BT} Fabrice Baudoin and Camille Tardif, \textit{Hypocoercive estimates on foliations and velocity spherical Brownian motion.} Preprint, \arXiv{1604.06813}.
\bibitem[BHLLO11]{BeL} C\'edric Bernardin, Fran\c{c}ois Huveneers,
Joel Lebowitz, Carlangelo Liverani and Stefano Olla, \textit{Green--Kubo
formula for weakly coupled systems with noise.} Comm. Math. Phys. 334(2015), no. 3, 1377-1412.
\bibitem[Bi05]{B}  Jean-Michel Bismut, \textit{The hypoelliptic Laplacian on the cotangent bundle.} J. Amer. Math. Soc. 18(2005), no. 2, 379-476.
\bibitem[Bi15]{B2}  Jean-Michel Bismut, \textit{Hypoelliptic Laplacian and probability.} J. Math. Soc. Japan 67(2015), no. 4, 1317-1357.
\bibitem[Bi11]{BL2} Jean-Michel Bismut, \textit{Hypoelliptic Laplacian and orbital integrals.} Annals of Mathematics Studies, 177. Princeton University Press, Princeton, NJ, 2011.
\bibitem[BiLe08]{BL1} Jean-Michel Bismut, Gilles Lebeau, \textit{The hypoelliptic Laplacian and Ray--Singer metrics.} Annals of Mathematics Studies, 167. Princeton University Press, Princeton, NJ, 2008.
\bibitem[DDZ14]{DDZ} Kiril Datchev, Semyon Dyatlov and Maciej Zworski, \textit{Sharp polynomial bounds on the number of Pollicott--Ruelle resonances,} Ergodic Theory Dynam. Systems 34 (2014), 1168-1183.
\bibitem[DoLi11]{DL} Dmitry Dolgopyat and Carlangelo Liverani, \textit{Energy transfer in a fast-slow Hamiltonian
system.} Comm. Math. Phys. 308(2011), no. 1, 201-225.
\bibitem[Dr15]{D} Alexis Drouot, Scattering resonances for highly oscillatory potentials. Preprint, \arXiv{1509.04198}.
\bibitem[Dr16]{D2} Alexis Drouot, Pollicott--Ruelle resonances via kinetic Brownian motion. Preprint, \arXiv{1607.03841v2}.
\bibitem[DSZ04]{DSZ} Nils Dencker, Johannes Sj\"ostrand and Maciej Zworski, \textit{Pseudospectra of semiclassical (pseudo)differential operators.} Comm. Pure Appl. Math., 57(2004), 384-415.
\bibitem[DyZw15]{DZ2} Semyon Dyatlov and Maciej Zworski, \textit{Stochastic stability of Pollicott--Ruelle resonances.} Nonlinearity 28(2015), no. 10, 3511-3533.
\bibitem[DyZw16a]{DZ1} Semyon Dyatlov and Maciej Zworski, \textit{Dynamical zeta functions via microlocal analysis.} Annales de l'ENS, 49(2016), 532-577.
\bibitem[DyZw16b]{DZ} Semyon Dyatlov and Maciej Zworski, \textit{Mathematical theory of scattering resonances.} Available online.
\bibitem[DyZw16c]{DZ3} Semyon Dyatlov and Maciej Zworski, \textit{Ruelle zeta function at zero for surfaces.} Preprint, \arXiv{1606.04560}.
\bibitem[El82]{El} David K. Elworthy, Stochastic differential equations on manifolds. London Mathematical Society Lecture Note Series, 70(1982).
\bibitem[FaSj11]{FS} Fr\'ed\'eric Faure and Johannes Sj\"ostrand, \textit{Upper bound on the density of Ruelle resonances for Anosov flows.} Comm. Math. Phys. 308(2011), 325-364.
\bibitem[GoLi06]{GoLi} Sebastien Gou\"ezel and Carlangelo Liverani, \textit{Banach spaces adapted to Anosov systems.} Ergodic Theory Dynam. Systems 26(2006), 189-217.
\bibitem[FrLe07]{LF} Jacques Franchi and Yves Le Jan, \textit{Relativistic diffusions and Schwarzschild geometry.}
Comm. Pure Appl. Math. 60(2007), 187-251.
\bibitem[Fr95]{Fr} David Fried, \textit{Meromorphic zeta functions for analytic flows.} Comm. Math. Phys. 174(1995),
161-190.
\bibitem[GaSh06]{GS06} S. V. Galtsev and A. I. Shafarevich, \textit{Quantized Riemann surfaces and semiclassical spectral series for a nonselfadjoint Schr\"odinger operator with periodic coefficients,} Theoret. and Math. Phys.
148(2006), 206-226.
\bibitem[GrSj94]{GS} Alain Grigis and Johannes Sj\"ostrand,  \textit{Microlocal analysis for differential operators. An introduction.} London Mathematical Society Lecture Note Series 196(1994).
\bibitem[GrSt13]{GS2} Martin Grothaus and Patrik Stilgenbauer, \textit{Geometric Langevin equations on submanifolds and applications to the stochastic melt-spinning process of nonwovens and biology.} Stoch. Dyn. 13 (2013), no. 4. 
\bibitem[H\"o67]{Ho} Lars H\"ormander, \textit{Second order differential equations.} Acta Math. 119(1967), 147-171.
\bibitem[Hs02]{Hsu} Elton P. Hsu, \textit{Stochastic analysis on manifolds.} 
Graduate Studies in Mathematics, 38 (2002).
\bibitem[J$\o$78]{J} E. J$\o$rgensen, \textit{Construction of the Brownian motion and the Ornstein-Uhlenbeck process in a Riemannian manifold on basis of the Gangolli-McKean injection scheme.} Z. Wahrsch. Verw.
Gebiete, 44(1978), 71–87.
\bibitem[Ko00]{K} Vassili N. Kolokoltsov, \textit{Semiclassical analysis for diffusions and stochastic processes.}  Lecture Notes in Mathematics 1724, Springer-Verlag, Berlin, 2000.
\bibitem[La08]{L08} Paul Langevin, \textit{Sur la théorie du mouvement brownien.} C. R. Acad. Sci. Paris 146(1908), 530-532.
\bibitem[Le07]{L}  Gilles Lebeau, \textit{\'Equations de Fokker-Planck g\'eom\'etriques. II. Estimations hypoelliptiques maximales.} Ann. Inst. Fourier (Grenoble) 57 (2007), no. 4, 1285-1314.
\bibitem[Li16]{Li} Xue-Mei Li,\textit{ Random perturbation to the geodesic equation.} Ann. Probab. 44(2016), no. 1, 544-566.
\bibitem[Li05]{Liv} Carlangelo Liverani, \textit{Fredholm determinants, Anosov maps and Ruelle resonances.} Discrete Contin.
Dyn. Syst. 13(2005), 1203-1215.
\bibitem[Me94]{Me} Richard B. Melrose, \textit{Spectral and scattering theory for the Laplacian on asymptotically Euclidian
spaces.} Spectral and scattering theory (M. Ikawa, ed.), Marcel Dekker, 1994.
\bibitem[NoZw15]{NZ} St\'ephane Nonnenmacher and Maciej Zworski, \textit{Decay of correlations for normally hyperbolic
trapping.} Invent. Math. 200(2)(2015), 345-438.
\bibitem[RoSt76]{RS} Linda Rothschild and Elias Stein, \textit{Hypoelliptic differential operators and nilpotent groups.} Acta Math. 137(1976), no. 3-4, 247-320.
\bibitem[Sh16]{S} Shu Shen, \textit{Analytic torsion, dynamical zeta functions and orbital integral.} Preprint, \arXiv{1602.00664}.
\bibitem[Sm16]{Sm} Hart Smith, Parametrix for a semiclassical sum of squares, in preparation.
\bibitem[So95]{So} M. R. Soloveitchik, \textit{Fokker-Planck equation on a manifold. Effective diffusion and spectrum.}
Potential Anal., 4(1995), 571-593.
\bibitem[Ts10]{Ts} Masato Tsujii, \textit{Quasi-compactness of transfer operators for contact Anosov 
ows.} Nonlinearity 23(2010), 1495-1545.
\bibitem[Va13]{Va} Andr\'as Vasy, \textit{Microlocal analysis of asymptotically hyperbolic and Kerr-de Sitter spaces, with an appendix by Semyon Dyatlov}. Invent. Math. 194(2013), 381-513.
\bibitem[Zw12]{Z}  Maciej Zworski, \textit{Semiclassical analysis.} Graduate Studies in Mathematics, 138(2012).
\end{thebibliography}
\end{document}